\documentclass[sn-mathphys-num]{sn-jnl}

\usepackage{graphicx}%
\usepackage{multirow}%
\usepackage{amsmath,amssymb,amsfonts,amssymb}%
\usepackage{amsthm}%
\usepackage{mathrsfs}%
\usepackage[title]{appendix}%
\usepackage{xcolor}%
\usepackage{textcomp}%
\usepackage[utf8]{inputenc}
\usepackage{manyfoot}%
\usepackage{booktabs}%
\usepackage{algorithm}%
\usepackage{algorithmicx}%
\usepackage{algpseudocode}%
\usepackage{listings}%
\usepackage{natbib}

\usepackage{multirow}
\RequirePackage{dsfont,mathabx}

\theoremstyle{thmstyleone}%
\newtheorem{theorem}{Theorem}
\newtheorem{corollary}{Corollary}
\theoremstyle{thmstyletwo}%
\newtheorem{lemma}{Lemma}%

\newtheorem{remark}{Remark}
\theoremstyle{thmstylethree}%
\newcommand{\ind}{\mathbf{1}}

\renewcommand{\tilde}{\widetilde}
\renewcommand{\hat}{\widehat}
\renewcommand{\check}{\widecheck}
\newcommand{\PP}{\mathbb{P}}

\newcommand{\Z}{\mathbb{Z}}
\newcommand{\E}{\mathbb{E}}

\newcommand{\1}{\mathbf{1}}
\newcommand{\R}{\mathbb{R}}
\newcommand{\N}{\mathbb{N}}

\newcommand{\V}{\mathbb{V}}

\newcommand{\eps}{\varepsilon}

\newcommand{\nrm}[1]{\lVert #1 \rVert}
\begin{document}

\title[Nonparametric density estimation for the small jumps of Lévy processes]{Nonparametric density estimation for the small jumps of Lévy processes}


%
%
%
\author{Céline Duval  \footnote{Sorbonne Université, LPSM, UMR 8001, F-75005 Paris, France.},
Taher Jalal \footnote{UMR 8100 - Laboratoire de Math\'ematiques de Versailles, UVSQ, CNRS,  Universit\'e Paris-Saclay, 78035 Versailles Cedex, France.}\  and \
Ester Mariucci \footnote{UMR 8100 - Laboratoire de Math\'ematiques de Versailles, UVSQ, CNRS, Universit\'e Paris-Saclay, 78035 Versailles Cedex, France.}}

\date{}


\abstract{We consider the problem of estimating the density of the process associated with the small
jumps of a pure jump Lévy process, possibly of infinite variation, from discrete observations of one trajectory. The interest of such a question lies on the observation that even when the Lévy measure is known, the density of the increments of the small jumps of the process cannot be computed in closed-form. We discuss results both
from low and high frequency observations. In a low
frequency setting, assuming the Lévy density associated with the jumps larger than $\eps\in (0,1]$ in absolute value is known, 
 a spectral  estimator relying on the convolution structure of the problem achieves a parametric
rate of convergence with respect to the integrated $L_2$ loss, up to a logarithmic factor. In a high frequency setting, we remove the assumption on the knowledge of the Lévy measure of the large jumps and show that the rate of convergence depends both on the sampling scheme and on the behaviour of the Lévy measure in a neighborhood of zero. We show that the rate we find is minimax up to a logarithmic factor.  An adaptive penalized procedure is studied to select the cutoff parameter. These results are extended to encompass the case where a Brownian component is present in the Lévy process. Furthermore, we illustrate numerically the performances of our procedures.
}

\maketitle
\noindent {\sc {\bf Keywords.}} {{Deconvolution, Lévy processes, Small jumps, Infinitely divisible distributions, Minimax rates of convergence}



\section{Introduction}

\subsection{Motivations}

Lévy processes are a class of jump processes that are particularly well-suited for modeling situations characterized by sudden and unpredictable changes. Initially, they gained prominence in mathematical finance and actuarial science due to their ability to capture the irregular behavior of financial markets and risk assessment. Over time, their applications have expanded to a variety of fields, including medicine and neuroscience, where they are used to model phenomena such as sudden neuronal firing or abrupt changes in biological systems.
Lévy processes are known for exhibiting diverse behaviors, such as heavy-tailed distributions, which make them useful for applications in climatology, seismology, and more recently, machine learning (see e.g. \cite{barndorff2012levy, boxma2011levy, Umut, noven2015levy} for reviews and other applications). Their capacity to model extreme events and rare occurrences further enhances their practical utility. Despite the broad range of behaviors they can exhibit, Lévy processes retain a rich mathematical structure that allows for rigorous theoretical analysis. At the foundation of more general jump processes, including the broader class of Itô semi-martingales, they are essential tools in both theoretical and applied stochastic analysis.

From a probabilistic point of view, the dynamics of the trajectories of a Lévy process $X$ is well understood. The law of $X$ is uniquely determined by the so-called Lévy triplet that contains a drift term, a diffusion coefficient and a Lévy measure  (see e.g. \cite{Bertoin,sato}). For any  pure jump L\'evy process $X$, the distribution of its increments is the convolution between the law of a martingale  $X^S$ describing its small jumps (i.e. of size less than any $\eps\in(0,1]$) and that of a compound Poisson process $X^{B}$ gathering the large jumps (larger than $\eps\in(0,1]$) of the process: $X=X^{S}+X^{B}$. For most Lévy processes with infinite Lévy measures, a closed-form expression for the law of its increments remains unknown. The main difficulty lies in computing the distribution of the small jumps, which is never available in closed-form even in very well known situations. For instance when $X$ is an $\alpha$-stable Lévy process, there exist many results for controlling the law of $X$ but nothing can be said for $X^S$ which is not an $\alpha$-stable Lévy process. 

In the literature, attempts have been made to circumvent the limited knowledge about the law of these processes by proposing approximations. Notably, the Gaussian approximation has emerged as a viable approach, showing promising results for Lévy processes with infinite activity (see, for example, \cite{asmussen2001, cohen2007gaussian, tankov, gnedenko1954limit, MR}). Although efficient in some cases, this approximation does have its limitations. In particular, its applicability tends to wane for high-frequency observations, where its accuracy may falter. The validity of the Gaussian approximation in total variation for small jumps of Lévy processes has been extensively studied in  \cite{Carpentier_2021}, where lower and upper bounds are established for the total variation distance between $n$ increments of $X^S$ and the nearest Gaussian vector. For example if $X$ is a symmetric $\alpha$-stable process, the total variation distance between the law of $(X_\Delta^S)^{\otimes n}$ and the nearest Gaussian vector tends to zero only if $\frac{\sqrt n \eps^\alpha}{\Delta}\to 0$. This means that if $\frac{\sqrt n \eps^\alpha}{\Delta}$ does not tend to 0,  it is possible to construct a test that allows distinguishing between observations from the jump model and those from a Gaussian vector. Statistically speaking, the two models are no longer (asymptotically) equivalent and the Gaussian approximation may not be  meaningful in these settings (see also Figure \ref{figure:Stable_NS} below for the unfitness of such approximation in the non-symmetric stable case).
This motivates the question of directly estimating the density of the increments of small jumps. Of course, the law governing small jumps remains fundamentally obscure, even under the assumption that the Lévy measure of the process is fully known. In the absence of closed-form formulas, a minimax estimator still provides valuable information about the shape of the true density and aids in gaining a better understanding of the regions in space where small jumps are more likely to occur.

More precisely, in this paper, we focus on the nonparametric estimation of the density $g_{\Delta}$ of an increment of the small jumps $X^{S}_{\Delta}$  of a Lévy process $X$ from $n$ observations collected with a sampling rate $\Delta$. So far there are no results in the literature focusing on the estimation of $X^S$, contrary to $X$ and $X^{B}$ which have been extensively studied (see e.g. \cite{belomestny2015levy,MR3091096,duval2019adaptive,duval2021spectral}). Using the convolution structure of the Lévy process and that  $X^B$ is a compound Poisson process with intensity and jump law depending on the Lévy measure of $X$ in an explicit way, we can derive an estimation procedure for the density of $X_{\Delta}^{S}$. \\

In estimating the density $g_{\Delta}$, a fundamental role is played by the behavior of the Lévy measure in a neighborhood of the origin. We do not require that the small jumps of $X$ are $\alpha$-stable, we only need a lower bound for the Lévy density in a neighborhood of the origin  (see \eqref{Ass:p} below). 
Under this assumption, both $X_\Delta$ and $X_\Delta^S$ have $C^{\infty}$ densities with all derivatives uniformly bounded (see Lemma \ref{lemmapicard}). In \cite{DJM}, a minimax estimator for the density $f_\Delta$ of $X_\Delta$ is proposed. In the present paper, we show that the rate of convergence for estimating the density $g_\Delta$ of $X_\Delta^S$ is essentially the same as that for estimating $f_\Delta$, for any $\Delta>0$. This result was not obvious: theoretically, the problem of estimating the density of small jumps is more complex than that of estimating $f_\Delta$ from direct observations of a sample with such a density. We can indeed consider the problem of estimating $g_\Delta$ as a deconvolution problem.
This work demonstrates that the complexity of inference on the density of discrete observations of a Lévy process is entirely driven by the complexity of inference for the small jumps.

More precisely, in the low frequency setting $\Delta>0$, we deal with a deconvolution problem  where the target density is super-smooth and to solve the problem we assume that the law of the large jumps of $X$ is known, which translates into knowing the Lévy measure of $X$ on $[-\eps,\eps]^c$. Again, we stress that this is not an oversimplifying context, even when the Lévy measure is known then we do have no access to a closed formula for the density of its small jumps. If $\Delta>0$, the inverse problem is well posed and as the density $g_{\Delta}$ is very regular our estimator attains, up to a logarithmic term, a parametric rate of convergence which is optimal  for the integrated $L_{2}$ loss (see \cite{butucea2004deconvolution}).

 In the high frequency setting $\Delta\to0$, without any knowledge on the distribution of $X^{B}$, another estimator can be proposed. Its rate of convergence is in ${(\log(n)/\Delta)^{{1}/{\alpha}}}/{n}$, it depends on the behavior of the Lévy density in a neighborhood of the origin (see Assumption \eqref{Ass:p}). Contrary to the case $\Delta$ fixed where a comparison with the deconvolution  literature is possible, the high frequency regime is utterly new. Thereby studying the optimality of the dependence in $\Delta$ of the rate for the estimation of $g_\Delta$ is necessary. Theorem \ref{thmLB} is a lower bound result that addresses this question and allows us to assert that up to a logarithmic factor, the rate found is indeed minimax. 
 
 All the rates of convergence have been identified through theoretical cutoffs that minimise the risk. However, these optimal cutoffs depend on the unknown  parameters of Assumption \eqref{Ass:p}. Therefore, a penalized adaptive procedure is proposed, adapting the one of \cite{comte2010nonparametric}. Finally, we extend our results to encompass the case where a Brownian component is present in the Lévy process, which depending on the behavior of $\Delta$, may alter importantly the rates of convergence. The performances of all these estimators are studied in  an extensive simulation study. In the remaining of this Section the principal notations and definitions are displayed.

\subsection{Setting and notations\label{sec:not}}
 Consider a pure jump L\'evy process $X$  characterized  by its L\'evy triplet $(\gamma_{\nu},0,\nu)$ where  $\nu$ is a Borel measure on $\R$ such that
\begin{align*}
\nu(\{0\})=0,\quad \nu(\R)=\infty\quad \textnormal{ and }\quad  \int_{\R}(x^2\wedge 1)\nu(d x)<\infty
\end{align*}  and 
\begin{equation*}\label{eq:bnu}\gamma_{\nu}:=\begin{cases}
             \int_{|x|\leq 1} x\nu(dx)\quad &\text{if}\quad \int_{|x|\leq 1}|x|\nu(dx)<\infty,\\
               0\quad &\text{if}\quad \int_{|x|\leq 1}|x|\nu(dx)=\infty.
             \end{cases}   
\end{equation*}
The Lévy-Khintchine formula gives the characteristic function of $X$ at any time $t\geq 0$:
\begin{align*}\label{LK}\phi_t(u)=\exp\bigg(it u \gamma_{\nu}+t\int_{\R}\big(e^{iux}-1-iux\1_{|x|<1}\big)\nu(dx)\bigg),\quad u\in\R.
\end{align*} 
Let $\eps\in(0,1]$ and let us consider pure jumps Lévy processes with a law that is absolutely continuous with respect to the Lebesgue measure. Thanks to the L\'evy-It\^o decomposition, $X$ can be written as
\begin{align*}
X_t=tb_{\nu}+X_t^S+X_t^B, \quad t\ge 0,
\end{align*}
where \begin{equation*}\label{eq:bnueps}b_{\nu}:=\begin{cases}
             \int_{|x|\leq \eps} x\nu(dx)\quad &\text{if}\quad \int_{|x|\leq 1}|x|\nu(dx)<\infty,\\
                -\int_{\eps< |x|\le 1}x\nu(dx)\quad &\text{if}\quad \int_{|x|\leq 1}|x|\nu(dx)=\infty,
             \end{cases}   
\end{equation*} $X^B$ is a compound Poisson process independent of $X^S$ with intensity $\lambda=\nu(\R\setminus[-\eps,\eps])$ and jump density $q=p\1_{[-\eps,\eps]^c}/\lambda$ with $p(x)=\frac{\nu(dx)}{dx}$, $X^S$ is a centered martingale accounting for the jumps of $X$ of size smaller than $\eps$, i.e. $$X^{S}_{t}=\lim_{\eta\to 0} \bigg(\sum_{s\leq t} \Delta X_s\1_{\eta<|\Delta X_s|\leq \eps}-t\int_{\eta<|x|\leq \eps} x\nu(dx)\bigg),$$ where $\Delta X_r$ denotes the jump at time $r$ of the c\`adl\`ag process $X$: $\Delta X_r=X_r-\lim_{s\uparrow r} X_s.$

In the following we  write $X_t^B=\sum_{i=1}^{N_t}Y_i$ where $N$ is a Poisson process of intensity $\lambda$ independent of the sequence of i.i.d. (independent and identically distributed)   random variables $Y_i$ with common density $q$. We will denote by $f_\Delta^B$ the density of $X_\Delta^B$ given by 
\begin{align}\label{eq:fB}
f_\Delta^B(x)=\sum_{k=0}^\infty \exp(-\lambda\Delta)\frac{(\lambda\Delta)^{k}}{k!}q^{*k}(x),
\end{align}
where $q^{*k}$ is the $k$-th convolution of the density $q$ and $q^{*0}=\delta_0$ is the Dirac measure at point 0. We refer to \cite{Bertoin} for an overview of the main properties of Lévy processes, including a thorough discussion of the Lévy-Khintchine formula and the Lévy-Itô decomposition.

Consider the i.i.d. observations $\mathbb X=(X_{i\Delta}-X_{(i-1)\Delta})_{i=1}^n$ with $X_0=0$. 
Our aim is to estimate the density $g_{\Delta}$ of $Z_\Delta:=X_\Delta^S+\Delta b_\nu$ from $\mathbb X$ both under the assumption $\Delta> 0$ fixed and $\Delta\to 0$, and compute the $L^2$ integrated risk. For that we need to assume that $X$ is a Lévy process with a Lévy density $p$ satisfying
\begin{align}\label{Ass:p} 
\int_{[-\eta,\eta]}x^{2}p(x)dx\geq {M}{\eta^{2-\alpha}}, \quad \forall0<\eta\leq \eps
\tag{$\mathcal A_{M,\alpha}$}\end{align}
for some $0<\alpha< 2$ and $M>0$. 

This means that we will consider Lévy processes whose Lévy density satisfies  Orey's condition. In Section 2.4 of \cite{DJM}, a discussion on the relationship between the parameter $\alpha$ and the Blumenthal-Getoor index of the process can be found. In most cases, these two quantities coincide, although it is possible to construct technical counterexamples where equality does not hold.
Under \eqref{Ass:p}, Lemma \ref{lemmapicard} below ensures that $g_{\Delta}$ lies in $L^{2}(\R)$ as its characteristic function is in $L^{1}(\R)$ and is bounded.

The estimation strategy that we analyse is based on a spectral approach, and we  use the following notations. Given a random variable $Z$, 
$\phi_Z(u)=\E[e^{iu Z} ]$ denotes the characteristic function of $Z$.  For $g\in L^1(\R)$, $\mathcal F g(u) =\int e^{iux } g(x) d x$ is the Fourier transform. Moreover, we denote by $\| \cdot  \|$ the $L^2$-norm of functions, $\| g\|^2:= \int |g(x)|^2 d x$.  Given some function $g\in L^1(\R) \cap L^2(\R)$,  we denote by $g_m$, $m>0$, the uniquely defined function with Fourier transform $\mathcal F g_m= (\mathcal F g)\mathds{1}_{[-m,m]}$. Finally, $\Gamma$ denotes the incomplete Gamma function $\Gamma(a,s)=\int_{s}^{\infty}t^{a-1}e^{-t}d t$, where $a>0,\, s\ge 0$.

\section{Main results}\label{main}
\subsection{Estimation in the low frequency regime}
Let $\Delta>0$ and suppose that $\nu$ is known on $\R\setminus[-\eps,\eps]$ such that in the decomposition: $X_{t}=b_{\nu}t+X^{S}_{t}+X^{B}_{t}=Z_t+X^{B}_{t}$ the density of $X^{B}_{t}$ is entirely known.  Thanks to the convolution structure of the law of $X_\Delta$, it holds
$
\phi_{X_{\Delta}} = \phi_{Z_{\Delta}} \phi_{X^{B}_{\Delta}}.$
As  for fixed $\Delta>0$ and $\eps\in(0,1]$, $\phi_{X^{B}_{\Delta}} $ is known and never vanishes  
\begin{align}\label{eq:minor}|\phi_{X^{B}_{\Delta}} (u)|= |\exp(\lambda\Delta(\phi_{Y_{1}}(u)-1)|\ge e^{-2\lambda\Delta}>0,\end{align}
 the quantity $$\phi_{Z_{\Delta}}(u) =\frac{ \phi_{X_{\Delta}}(u)}{ \phi_{X^{B}_{\Delta}}(u)} $$ is well defined for all $u\in \R$. It can be estimated by
\begin{align}\label{eq:phiZ}
\hat \phi_{Z_{\Delta}}(u) =\frac{ 1}n\sum_{j=1}^{n}\frac{e^{iu(X_{j\Delta}-X_{(j-1)\Delta})}}{ \phi_{X^{B}_{\Delta}}(u)} .
\end{align} 
Let $m>0$, from \eqref{eq:phiZ} we derive an estimator of $g_{\Delta}$,  using a spectral cut-off as the latter quantity may not be in $L^{1}(\R)$: 
\begin{align}
\label{eq:estFg}\hat g_{{\Delta},m}(x) =\frac{ 1}{2\pi}\int_{-m}^{m}\hat \phi_{Z_{\Delta}}(u)e^{-iux}du.
\end{align}
The following result gives an upper bound for the integrated $L^{2}$-risk of $\hat g_{{\Delta},m}$.
\begin{theorem}\label{thm1}
Let $X$ be a Lévy process with a Lévy measure $\nu$ that satisfies \eqref{Ass:p}, for some $M>0$ and $\alpha\in(0,2)$. Let  $\Delta>0$, $\eps\in(0,1]$ and $g_\Delta$ be the density of $\Delta b_\nu+X_\Delta^S$ and $\hat g_{\Delta,m}$ the estimator defined in \eqref{eq:estFg}. Then,  for all $m\geq \pi/(2\eps)$ it holds that
$$\E[\|\hat{g}_{\Delta,m}-g_{\Delta}\|^2]\leq   \|g_{\Delta,m}-g_{\Delta}\|^2+\frac{e^{4\lambda\Delta}}{\pi} \frac mn , $$ and $ \|g_{\Delta,m}-g_{\Delta}\|^2\le  C\Delta^{-\frac1\alpha}\Gamma\left({1}/{\alpha},c\Delta m^{\alpha}\right)$ for  constants $c>0$ and $C>0$ depending on $\alpha$ and $M$ given in \eqref{eq:biasF}. \end{theorem}
\begin{proof}
To control the integrated $L^{2}$-risk
 we write the decomposition \begin{align*}\E[\|\hat{g}_{\Delta,m}-g_{\Delta}\|^2]&=\|g_{\Delta,m}-g_{\Delta}\|^2+\E[\|\hat{g}_{\Delta,m}-g_{\Delta,m}\|^2]\\ & = \|g_{\Delta,m}-g_{\Delta}\|^2+\frac{1}{2\pi}\int_{- m}^{ m} \E[|\hat\phi_{Z_{\Delta}}(u)-\phi_{Z_{\Delta}}(u)|^2]d u.\end{align*}   The first term is the standard bias term for which we can write using Plancherel equality, Lemma \ref{lemmapicard}, the fact that $m\geq\pi/(2\eps)$ and  \eqref{Ass:p}, that
 \begin{align}
 \|g_{\Delta,m}-g_{\Delta}\|^2&=\frac1{2\pi}\int_{[-m,m]^{c}}|\phi_{Z_{\Delta}}(u)|^{2}du\le \frac{1}{\pi}\int_{m}^{\infty}e^{-\frac{2^{\alpha+1}M}{\pi^{\alpha}}u^{\alpha}\Delta}du \nonumber\\ &=\frac{1}{2\alpha(2M\Delta)^{\frac1\alpha}}\Gamma\left(\frac{1}{\alpha},\frac{2^{\alpha+1}M\Delta m^{\alpha}}{\pi^{\alpha}}\right).  \label{eq:biasF}
 \end{align}
 For the variance term, using that 
 \begin{align*}\E[|\hat\phi_{Z_{\Delta}}(u)-\phi_{Z_{\Delta}}(u)|^2]&=\frac{1}{|\phi_{{X^{B}_{\Delta}}}(u)|^{2}}\E[|\hat\phi_{X_{\Delta}}(u)-\phi_{X_{\Delta}}(u)|^2]\\&=\frac{1}{|\phi_{{X^{B}_{\Delta}}}(u)|^{2}}\V\left(\frac1n\sum_{j=1}^{n}e^{iu(X_{j\Delta}-X_{(j-1)\Delta})}\right)= \frac{1-|\phi_{{X_{\Delta}}}(u)|^{2}}{|\phi_{{X^{B}_{\Delta}}}(u)|^{2}}\frac1n,
  \end{align*} we easily get with \eqref{eq:minor} that
 $\E[\|\hat{g}_{\Delta,m}-g_{\Delta,m}\|^2]\le \frac{e^{4\lambda\Delta}}{\pi}\frac{m}{n}.$ Gathering both inequalities completes the proof.
\end{proof}
\begin{remark}\label{rmq1} To find the value $m^{\star}$ that minimizes the bound in Theorem \ref{thm1} we differentiate this bound in $m$ using \eqref{eq:biasF}. If $\lambda\Delta=O(1),$ and $n$ is such that $\log n\ge 4\lambda\Delta$ we find that $m^{\star}$ is solution of 
$$ ne^{-4\lambda\Delta}=e^{\frac{2^{\alpha+1}M}{\pi^{\alpha}}\Delta {m^{\star}}^{\alpha}}\quad \Longleftrightarrow\quad m^{\star}=\frac\pi2\left(\frac{\log n-4\lambda\Delta}{2M\Delta}\right)^{\frac1\alpha}.$$ With this optimal cutoff, the rate implied by Theorem \ref{thm1} is
$$\E[\|\hat{g}_{\Delta,m^{\star}}-g_{\Delta}\|^2]\le \frac{1}{2\alpha(2M\Delta)^{\frac1\alpha}}\Gamma\left(\tfrac{1}{\alpha},\log n-4\lambda\Delta\right)+\frac1{2(2M)^{\frac{1}{\alpha}}}\left(\frac{\log n}{\Delta}\right)^{\frac1\alpha}\frac{e^{4\lambda\Delta}}{ n}.$$ Using that $\Gamma(a,s)\sim e^{-s}s^{a-1} $ as $s\to\infty$, we get that   \begin{align}\label{eq:rate}\E[\|\hat{g}_{\Delta,m^{\star}}-g_{\Delta}\|^2]\le C\left(\frac{\log n}{\Delta}\right)^{\frac1\alpha}\frac{e^{4\lambda\Delta}}{ n}\end{align}
 for some positive constant $C$. This is an almost (up to a logarithmic factor) parametric rate (recall that $\Delta>0$ is fixed), which is consistant with the fact that: i)
 we are in a well posed deconvolution problem (see \eqref{eq:minor}),
ii)
 under the assumptions of Lemma \ref{lemmapicard}, $g_{\Delta}$ is $C^{k}(\R)$ for all $k\in\N$. Note that if $\Delta$ goes rapidly to 0 (e.g. if  $\Delta=\log\log n/n$ and $\alpha=1$) the upper-bound \eqref{eq:rate} does not tend to 0 because of the logarithmic term. Nonetheless, \cite{10.1214/aos/1176347736} (Example 2) and \cite{butucea2004deconvolution} seem to indicate that for fixed $\Delta$ the logarithmic term $(\log n)^{1/\alpha}$ is optimal. \end{remark}


The problem of finding a data driven way to select $m$ is studied in Section \ref{sec:adapt}. The optimal cutoff $m^{\star}$ depends on the unknown quantity $\alpha$ appearing in Assumption \eqref{Ass:p}.
 Interestingly, the adaptation problem of selecting $m$ consists in estimating a possible $\alpha$ for condition \eqref{Ass:p}. If  \eqref{Ass:p} is satisfied for a given $\alpha_2$, it will also be satisfied for any $\alpha_1<\alpha_2$. The choice of the maximum $\alpha$ such that the hypothesis is satisfied is however important in regimes where $\Delta<1$, indeed the function $\alpha\mapsto \Delta^{-1/\alpha}$ is decreasing, therefore to attain the optimal rate the largest value of $\alpha$ such that \eqref{Ass:p} is satisfied should be selected (see also Theorem \ref{thmLB}).

\subsection{Another strategy in the high frequency regime}\label{HF}
In this section we consider the case where $\Delta \to 0$. Despite this limit, it remains feasible to estimate the density of $Z_\Delta$ using the estimator $\hat g_{{\Delta},m}$ as defined in \eqref{eq:estFg}. Employing similar arguments to those discussed in the preceding paragraph, one can demonstrate its consistency as long as ${n\Delta^{{1}/{\alpha}}}/{(\log n)^{1/\alpha}}\to \infty$, and its $L^2$ rate of convergence is still $n^{-1}\Delta^{-{1}/{\alpha}}$ up to a log factor.

In the high frequency setting, it is possible to omit the assumption that $\phi_{X^{B}_{\Delta}}$ is known since in this asymptotic $\phi_{X^{B}_{\Delta}}$ is close to 1. We therefore propose to consider a second estimator of $g_{\Delta}$,  defined as follows \begin{align}
\label{eq:estFg0}\tilde g_{{\Delta},m}(x) =\frac{ 1}{2\pi}\int_{-m}^{m} \hat \phi_{X_{\Delta}}(u)e^{-iux}du.
\end{align}
Note that if $\Delta$ is fixed, \eqref{eq:estFg0} is an estimator of the density of $X_{\Delta}$ (see Section 4 of \cite{kappus2015nonparametric} and \cite{DJM}).

\begin{theorem}\label{thm2}
Let $X$ be a Lévy process with a Lévy measure $\nu$ that satisfies \eqref{Ass:p}, for some $M>0$ and $\alpha\in(0,2)$. Let  $\Delta\in(0,1)$ and $\eps\in(0,1]$ be such that $\lambda\Delta\leq 1$, where $\lambda=\nu(\R\setminus[-\eps,\eps])$. Then, there exist $K,\,\kappa>0$ depending on $\alpha$ and $M$ such that for all $m\geq 1$ it holds:
$$\E[\|\tilde{g}_{\Delta,m}-g_{\Delta}\|^2]\leq K\left( \|g_{\Delta,m}-g_{\Delta}\|^2+ \frac mn + \lambda^2\Delta^{2-\frac{1}{\alpha}}e^{-\kappa\frac{\Delta}{\eps^{\alpha}}}+\lambda^2\Delta^{2}\eps^{-1}\right),$$
and $ \|g_{\Delta,m}-g_{\Delta}\|^2\le C\Delta^{-\frac1\alpha}\Gamma\left({1}/{\alpha},c\Delta m^{\alpha}\right),$ for  constants $c>0$ and $C>0$ depending on $\alpha$ and $M$ given in \eqref{eq:biasF}. 
\end{theorem}

 Computations developed in Remark \ref{rmq1} remain valid to realise the bias-variance tradeoff between the first two terms in the above upper bound. It follows that the rate of convergence of $\tilde{g}_{\Delta,m}$, choosing $m^{\star}$ as in Remark \ref{rmq1} and if $\Delta\eps^{-\alpha}\le 1$, is of order of
 \begin{align}\label{rate2}
 \max\left\{\frac{(\log n)^{\frac1\alpha}}{n\Delta^{\frac1\alpha}}, \Delta^{2-\frac{1}{\alpha}}\right\}
 \end{align}which is of order of $(\log n)^{1/\alpha}({n\Delta^{1/\alpha}})^{-1}$ if $\alpha>{1}/{2}$. We also underline the fact that the condition $\Delta\lambda\le 1$ is equivalent to $\Delta\eps^{-\alpha}\le 1$ for instance for $\alpha$-stable Lévy processes.
Furthermore we notice that for $\alpha\leq {1}/{2}$, the consistency of $\tilde{g}_{\Delta,m}$ is not ensured. 
Finally, we observe that it is always possible to estimate $g_\Delta$ with a rate of order of $(\log n)^{1/\alpha}({n\Delta^{1/\alpha}})^{-1}$ for any $\alpha\in (0,2)$ by means of the estimator $\hat{g}_{\Delta,m}$ defined in \eqref{eq:estFg}. However, such an estimator  requires the knowledge the law of $X_\Delta^B$, whereas this assumption is not needed to define $\tilde{g}_{\Delta,m}$.

\subsection{Comments}
More generally, a natural question is the necessity of assuming knowledge of the law of $X_\Delta^B$. In Section \ref{HF}, taking advantage of the high frequency regime, the contribution of the process $X^{B}_{\Delta} $ has been ignored at the cost of the term $\Delta^{2-1/\alpha}$ in the bound \eqref{rate2} imposing the constraint $\alpha\ge 1/2$. In a high-frequency regime such that $n\Delta^2\to 0$, it may be possible to loosen such an assumption. Indeed, then it holds that (using that $\log(1+x)\ge x-\tfrac{x^{2}}2$ for $|x|<1$)  $$\PP(\forall i\in\{1,\ldots,n\},\, N_{i\Delta}-N_{(i-1)\Delta}\in\{0,1\})\ge e^{-\frac12n(\lambda\Delta)^{2}}\ge 1-\frac12n(\lambda\Delta)^2,$$ informally increments of $X$ larger than $\eps$ in absolute value can be considered as realizations of $Y_{1}$ defined in Section \ref{sec:not} (as studied in \cite{duval2021spectral}). Therefore, the quantity $\phi_{X_\Delta^B}$ can be estimated from the increments $X_{\Delta}$ such that $X_{\Delta}\ge \eps$ allowing to estimate the characteristic function of $Y_{1}$ and therefore $\phi_{X^{B}_{\Delta}}$ using that $ \phi_{X^{B}_{\Delta}} = \exp(\lambda\Delta(\phi_{Y_{1}}-1))$. 
Still, computations become considerably longer and tedious, but they should not affect the convergence rate.

The low-frequency case poses a significantly greater challenge. One possibility would involve estimating $\phi_{X_\Delta^B}$ via a plug-in of an estimator of the Lévy density. 
More precisely, we note that a strategy that does not require knowledge of the Lévy measure $\nu$ could be formulated as follows:
\begin{enumerate}
    \item
    Estimate $p(x) = d\nu(x)/dx$, the Lévy density (see e.g., \cite{comte2010nonparametric}). For instance, if $\int (x \wedge 1) \nu(dx) < \infty$, $\Delta > 0$ is fixed and $\eps=1$ for simplicity, one could use:
    \begin{equation*}
        \hat{p}_{m_{1}}(x) = \frac{1}{2\pi x} \int_{-m_{1}}^{m_{1}} \hat{\Psi}_{\Delta}(u)e^{-iux}du,
    \end{equation*}
    where 
    \begin{equation*}
        \hat{\Psi}_{\Delta}(u) = \frac{\frac{1}{n}\sum_{j=1}^{n} X_{j,\Delta} e^{iuX_{j,\Delta}}}{\Delta \frac{1}{n} \sum_{j=1}^{n} e^{iuX_{j,\Delta}}}, \quad X_{j,\Delta} = X_{j\Delta} - X_{(j-1)\Delta}.
    \end{equation*}
 \item
    Estimate the characteristic function of the increment of a small jump by:
    \begin{equation*}
        \hat{\Phi}_{Z_{\Delta}, m_{1}}(u) = \exp\left(\Delta \int_{-1}^{1} (e^{iuy} - iuy - 1) \hat{p}_{m_{1}}(y) dy\right).
    \end{equation*}
\item
    Derive an estimator of $g_{\Delta}$ as follows:
    \begin{equation*}
        \tilde{g}_{\Delta, m_{1}, m_{2}}(x) = \frac{1}{2\pi} \int_{-m_{2}}^{m_{2}} e^{-ivx} \hat{\Phi}_{Z_{\Delta}, m_{1}}(v) dv.
    \end{equation*}
\end{enumerate}
The theoretical analysis of this estimator raises significant challenges as it  involves controlling two bias terms and a tricky variance term where  $\hat{p}_{m_{1}}$ appears under an integral within an exponential term. 
Furthermore, there are two cutoffs, $m_{1}$ and $m_{2}$, for which determining an adaptive choice would be non-trivial. Numerically, this method is also computationally expensive due to the need for approximating three separate integrals.
A related strategy was adopted in Section 4 of \cite{kappus2015nonparametric} to estimate the density of $X_{\Delta}$ for $\Delta=1$. Even in that direct context, it is unclear whether it  leads to minimax results, and would require different assumptions on the process than those adopted in the present work. This approach lies outside the scope of this paper's objectives.

\subsection{Lower bound result}

In this section, we show that the convergence rates of our estimators are minimax, up to a logarithmic factor. The proof relies on establishing that the estimation of
the density of the small jumps $g_{\Delta}\in (\mathcal A_{M,\alpha_0})$ is an harder problem than estimating the density of a symmetric $\alpha_0$-stable process from discrete direct observations of the process at sampling rate $\Delta$, i.e.
  \begin{align*}\inf_{\hat g_{\Delta}}\sup_{g_{\Delta}\in \mathcal A_{M,\alpha_0}}\E\|\hat g_{\Delta}-g_{\Delta}\|^{2}\geq \inf_{\hat f_{\Delta}}\sup_{f_{\Delta}\in {\mathcal S}_{\alpha_0}}\E\|\hat f_{\Delta}-f_{\Delta}\|^{2},
  \end{align*}
where $\mathcal S_{\alpha_0}$ denotes the class of symmetric $\alpha$-stable densities with $\alpha\in[\alpha_0,2)$ and characteristic function given by $e^{-\Delta|u|^{\alpha_0}}$ and the second infimum is taken over all possible estimators $\hat f_\Delta$ of the density $f_\Delta$ of $X_{\Delta}$. Then, we leverage the lower bound derived for the direct estimation problem in Theorem 3 of \cite{DJM}.  
 \begin{theorem}\label{thmLB}
Let $\alpha_0\in(0,2)$, $0<M\le  F({\alpha_{0}})$ where $F$ is defined in \eqref{eq:MajM}, and $0<\Delta<  e^{-\frac{4}{2-\alpha_0}}$. There exists $K>0$ such that  for any $n$ satisfying $n\log^2(\Delta)\geq K$ it holds
$$\inf_{\hat g_{\Delta}}\sup_{g_{\Delta}\in \mathcal A_{M,\alpha_0}}\E\|\hat g_{\Delta}-g_{\Delta}\|^{2}\geq \frac{c_{0}}{n\Delta^{\frac{1}{\alpha_0}}},$$
 where the infimum is taken over all possible estimators $\hat g_{\Delta}$ of $g_{\Delta}$ and $c_{0}$ is a strictly positive constant only depending on $\alpha_0$. 
\end{theorem}


Theorem \ref{thmLB} above allows us to assert that the rates found in Theorems \ref{thm1} and \ref{thm2} are nearly minimax. This conclusion was far from obvious. While it was evident that the rate in $n$ in \eqref{eq:rate} and \eqref{rate2} could not be improved, as it already represents a parametric rate, the dependence of the rate on $\Delta$ and $\alpha$ was less clear. For fixed $\Delta$, the problem remains relatively straightforward, but it becomes significantly more intricate in the case of high-frequency observations. To the best of our knowledge, this is the first result establishing minimax optimality for estimating the density of the law of small jumps of Lévy processes. 
Furthermore, since \cite{DJM} demonstrates that the minimax rate for estimating the density $f_\Delta$ is the same as the one derived in this work for estimating $g_\Delta$, this indicates that the complexity of inference for a Lévy process observed at high frequency is essentially driven by the small jumps.

\subsection{Adaptation procedure\label{sec:adapt}}

We propose an adaptive procedure to select $m$ for  the estimator $\hat{g}_{\Delta,m}$ defined in \eqref{eq:estFg} that enables to attain the bound of Theorem \ref{thm1}. This procedure is a penalization procedure inspired by the one proposed in \cite{comte2010nonparametric}. Note that it can be straightforwardly adapted to select $m $ for  the estimator $\tilde{g}_{\Delta,m}$ defined in \eqref{eq:estFg0}.

 Consider the space $S_{m}=\{t\in L^{2}(\R),\ \mbox{supp}(\mathcal{F}(t))\subset [-m,m]\}.$ This space is generated by an orthonormal basis defined by
\begin{align}\label{eq:basis}\psi_{m,j}(x)=\sqrt{\pi m}\psi(mx-j),\quad j\in\Z\quad \psi(x)=\frac{\sin(x)}{\pi x}.\end{align}
 Indeed $\mathcal{F}\psi_{m,j}(u)=\sqrt{\pi}\frac{e^{iu j/m}}{\sqrt{m}}\mathbf{1}_{[-m,m]}(u)$ and it holds using Plancherel $$\langle\psi_{m,j},\psi_{m,k}\rangle=\frac{1}{2\pi}\langle\mathcal{F}\psi_{m,j},\mathcal{F}\psi_{m,k}\rangle=\frac{1}{2 m}\int_{-m}^{m}e^{\frac{iu}m(j-k)}d u=\delta_{jk}.$$ Therefore, we have the following decomposition $$\hat g_{\Delta,m}=\sum_{j\in\Z}\hat a_{m,j}\psi_{m,j},\quad \hat a_{m,j}=\langle \hat g_{\Delta,m},\psi_{m,j}\rangle=\frac{\sqrt{\pi}}{2\sqrt{m}}\int_{-m}^{m} \hat \phi_{Z_{\Delta}}(u)e^{-\frac{iuj}{m}} du.$$
 Using either Plancherel or this series representation, we get
 $$\|\hat g_{\Delta,m}\|^{2}=\frac{1}{2\pi}\int_{-m}^{m}|\hat \phi_{Z_{\Delta}}(u)|^{2}du=\sum_{j\in \Z}|\hat a_{m,j}|^{2}.$$
 
The adaptive procedure is built using penalization techniques. We define the contrast for $t\in S_{m}$, 
$$\gamma_{n}(t)=\| t\|^{2}-2\langle \hat g_{\Delta,m}, t\rangle=\| t\|^{2}-\frac1\pi\int \hat\phi_{Z_{\Delta}}(u)\mathcal Ft(-u)du$$ for which we easily check that $\hat g_{\Delta,m}=\mbox{arg}\min_{t\in S_{m}}\gamma_{n}(t)$ and $\gamma_{n}( \hat g_{\Delta,m})=-\|\hat g_{\Delta,m}\|^{2}$. Con\-si\-de\-ring a collection $(S_{m}, m=\lceil\tfrac{\pi}{2\eps}\rceil,\ldots,n)$  we select adaptively $m$ satisfying
\begin{align}\label{eq:adap}
\hat m=\mbox{arg}\min_{m\in \{1,\dots,n\}}\left(\gamma_{n}(\hat g_{\Delta,m})+\mbox{pen}(m)\right),\quad \mbox{with pen}(m)=\kappa e^{4\lambda\Delta}\frac mn.
\end{align}

\begin{theorem}\label{thm:adapt} Under the assumptions of Theorem \ref{thm1}, the adaptive estimator $\hat g_{\Delta, \hat m}$ defined in \eqref{eq:estFg} with $\hat m$ defined in \eqref{eq:adap} for $\kappa>32/(3\pi)$ satisfies for a positive constant $C$
\begin{align*}
\E[\|\hat g_{\Delta,\hat m}-g_{\Delta}\|^{2}]&\le 3\inf_{m\in\{1,\ldots,n\}}\left(\E[\|\hat g_{\Delta,m}-g_{\Delta}\|^{2}]+{\rm pen}(m)\right)+ \frac{C}{n}.\end{align*}  
\end{theorem}

Theorem \ref{thm:adapt} ensures that under \eqref{Ass:p} the adaptive estimator $\hat g_{\Delta,\hat m}$ attains the optimal rate $(\log n)^{1/\alpha}/(n\Delta^{1/\alpha})$ of convergence.

\subsection{Estimation in presence of a Brownian component}

A natural question is whether the above results hold true  for general L\'evy processes, that is in presence of a Gaussian part. Let $\sigma >0$, the L\'evy-It\^o decomposition of a L\'evy process $X$ of L\'evy triplet $(\gamma_{\nu},\sigma^{2},\nu)$ allows to write $X$ as the sum of four independent L\'evy processes: 
\begin{align}\label{eq:Xito2}
 X_t&=: tb_{\nu}+\sigma W_{t}+X^{S}_t+X^{B}_t,\quad \forall t\geq 0,
\end{align}
where $ W$ is a standard Brownian motion. The convolution structure of the model is preserved and the latter strategy can be adapted assuming that $\sigma$ and the density of $X_{\Delta}^{B}$ are known. Nonetheless, we expect deteriorated rates of convergence as we face a deconvolution problem with a Gaussian error (see \cite{lacour2006rates} and \cite{MR2354572,MR2742504}).  For $\Delta>0$, we have the representation
$X_\Delta= Z_\Delta + X_\Delta^B + \sigma W_\Delta,$ and we now estimate the characteristic function of the small jumps by (see \eqref{eq:phiZ})
\begin{align}\label{form:estimator_gDelta_with_Brw}
\widecheck{\phi}_{Z_\Delta}(u) = \frac{1}{n} \sum_{j=1}^{n} \frac{e^{iu(X_{j\Delta}-X_{(j-1)\Delta})}}{\phi_{X_{\Delta}^B}(u) \phi_{\sigma W_\Delta}(u)},
\end{align}
where $\phi_{\sigma W_\Delta}(u)= e^{-\sigma^2\Delta \frac{u^2}{2}}.$ Using a Fourier inversion and a cut-off, we derive the following estimator
\begin{align}\label{eq:estFgB}
\check{g}_{\Delta,m}(x) = \frac{1}{2\pi} \int_{-m}^m \check{\phi}_{Z_\Delta}(u) e^{-iux}dx.
\end{align}
Theorem \ref{thm:MB} below provides an upper bound for the integrated $ L^2$ risk of $\check{g}_{\Delta,m}$. 

\begin{theorem}\label{thm:MB}
Let $X$ be a Lévy process with triplet $(\gamma_\nu,\sigma^2,\nu)$ and with a Lévy measure $\nu$ that satisfies \eqref{Ass:p}, for some $M>0$ and $\alpha\in(0,2)$. Let  $\Delta>0$, $\eps\in(0,1]$, $g_\Delta$ be the density of $\Delta b_\nu+X_\Delta^S$ and $\check g_{\Delta,m}$ the estimator defined in \eqref{eq:estFgB}. Then, for all $m\geq \pi/(2\eps)$ it holds that
$$\E[\|\check{g}_{\Delta,m}-g_{\Delta}\|^2]\leq   \|g_{\Delta,m}-g_{\Delta}\|^2+\frac{e^{4\lambda\Delta}}{\pi} \frac {\int_{0}^{m\sqrt{\Delta}\sigma} e^{z^2}dz}{n\sqrt{\Delta}\sigma},$$ and $ \|g_{\Delta,m}-g_{\Delta}\|^2\le  C\Delta^{-\frac1\alpha}\Gamma\left(\tfrac{1}{\alpha},c\Delta m^{\alpha}\right)$ for  constants $c>0$ and $C>0$ depending on $\alpha$ and $M$ given in \eqref{eq:biasF}.
\end{theorem}

Deriving convergence rates in this framework is intricate, firstly because no closed-form formula of the optimal choice of $m^{\star}$ is available (see \cite{lacour2006rates}), and secondly because we are interested in several asymptotic depending in the behaviour of $(n,\Delta,\sigma)$. For the sake of simplicity, we  only consider the case  $\alpha=1$, for which explicit computations can be carried out, and specific asymptotic for $(\Delta,\sigma)$ that allows to  recover known rates. 

\begin{corollary}\label{cor:rateB}
Under the assumptions of Theorem \ref{thm:MB}, assuming that $\alpha=1$, it holds that 
$$\E[\|\check{g}_{\Delta,m^{\star}}-g_{\Delta}\|^2]\leq  \begin{cases}
K e^{-\kappa \sqrt{\log n}}& \mbox{ if } \Delta \mbox{ and } \sigma \mbox{ are fixed},\\
 K'\frac{\log n}{n\Delta}&\mbox{ if } \sigma^{2}\Delta^{-1} (\log n)^{2}\le 1,
\end{cases} $$  for positive constants $K$ and $\kappa$ depending on $\Delta,\sigma, M, \lambda$, and for a universal positive constant $K'$.
\end{corollary}

The first rate in Corollary \ref{cor:rateB} corresponds to the classical deconvolution rate in a Gaussian framework derived e.g. in \cite{lacour2006rates} and \cite{MR2354572,MR2742504}.
 The second case corresponds to a setting where $\sigma$ goes to 0 rapidly enough so that the rate is not affected by the presence of a Brownian part, we recover the rate of Theorems \ref{thm1} and \ref{thm2}.

\section{Numerical examples}

\subsection{Setting}

We illustrate the numerical performances of the adaptive estimator  $\hat g_{\Delta, \hat m}$ defined in \eqref{eq:estFg} with $\hat m$ defined in \eqref{eq:adap}. We fix $\eps=1$ and consider  Lévy processes $X$ with Lévy density of the form
\begin{align}\label{form:TSP}
\frac{\nu(dx)}{dx} = \frac{P}{x^{1+\alpha}} e^{-Ax} \ind_{x>0} + \frac{Q}{|x|^{1+\alpha}} e^{-B|x|} \ind_{x<0},
\end{align}
where $P,Q,A,B$ are non-negative constants and $0<\alpha<2$. Note that the assumption \eqref{Ass:p} is met for $(M,\alpha)=\left( \frac{Pe^{-A} + Qe^{-B}}{2-\alpha} ,\alpha\right)$. The case $A=B=0$ corresponds to an $\alpha$-stable Lévy process, otherwise the process $X$ is a tempered stable Lévy process. 
To simulate the increments of an $\alpha$-stable process we use the self-similarity property, \textit{i.e.} $X_\Delta \sim \Delta^{1/\alpha} X_1$ where $X_{1}$ is distributed as an $\alpha$-stable random variable (see Masuda \cite{belomestny2015levy} or \cite{samorodnitsky_stable_1994}).
Tempered stable Lévy processes do not exhibit self-similarity, to simulate such processes we use a compound Poisson approximation approach described in Example 6.9 of \cite{tankov} (see also Section 4.5 therein).

It is not straightforward to evaluate the associated $L^2$ error of the estimator $\hat{g}_{\Delta,\hat{m}}$ since no closed-form formula for $g_\Delta$ is available, even when the Lévy density is known. However, for $\alpha$-stable and tempered stable Lévy processes, we can derive a useful expression of the characteristic function of $g_\Delta$ thanks to the Lévy-Khintchine formula
\begin{align}
        \phi_{Z_\Delta}(u) = \begin{cases}
            e^{\Delta \left[\int_{0}^1 \frac{\cos(ux)-1}{x^{1+\alpha}}\left(Pe^{-Ax} + Qe^{-Bx}\right)dx + i\int_{0}^1 \frac{\sin(ux)}{x^{1+\alpha}}\left(Pe^{-Ax} - Qe^{-Bx}\right)dx\right] }, & \alpha<1,\\
             e^{\Delta\left[\int_{0}^1 \frac{\cos(ux)-1}{x^{1+\alpha}}\left(Pe^{-Ax} + Qe^{-Bx}\right)dx + i \int_{0}^1 \frac{\sin(ux)-ux}{x^{1+\alpha}}\left(Pe^{-Ax} - Qe^{-Bx}\right)dx\right] }, & \alpha \geq 1.
        \end{cases}
    \end{align}
By numerical Fourier inversion we approximate $g_\Delta(x)$ by $g_{\Delta,\ell}(x):=\int_{-\ell}^{\ell} \phi_{Z_\Delta}(u)e^{-iux}du$ which is used as a benchmark to compute the $L^2$ loss. In practice we select $\ell$ large enough such that $g_{\Delta,\ell}$ does not change, i.e. for $\eta$ small enough $|g_{\Delta,\ell}(x)-g_{\Delta,\ell'}(x)|\le \eta$ for all $x$ and  $\ell'\ge \ell$. After preliminary simulation experiments, we select $\ell=1000$ when $\Delta<1$ and $\ell=100$ when $\Delta=1$ in the $\alpha-$stable case and $\ell=50$ when $\Delta<1$ and $\ell=10$ when $\Delta=1$ in the tempered stable case. To help the comparison between the different examples, where $\|g_\Delta\|^2$ may vary a lot, we compute the relative $L^2$ error defined as $\frac{\|{\hat{g}_{\Delta, \hat{m}} - g_\Delta}\|^2}{\|{g_\Delta}\|^2}$.  The calibration of the constant $\kappa$ in the penalty term is also done by preliminary simulation experiments. This constant is selected as $\kappa=0.9$. We compute a Monte Carlo estimate with $100$ values of the $L^2$ risk for the different examples.

Hereafter, we illustrate our procedure both visually and by comparing their risk for different examples, both stable and tempered, and different values of  $n \in \{500,1000,10000\}$, $\Delta \in \{0.01,0.1,1\}$ and $\alpha \in \{0.7,1.1,1.7\}$.

\subsection{Results and comments} 

\begin{figure}[h!]
    \centering   
   \includegraphics[scale=0.3]{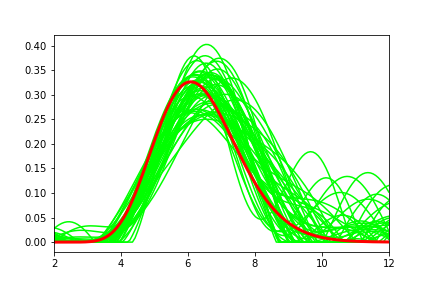}\includegraphics[scale=0.3]{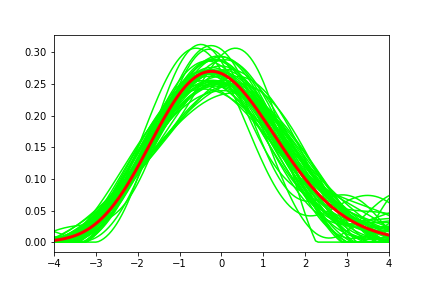}\includegraphics[scale=0.3]{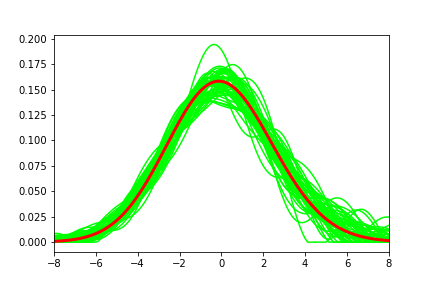}\\
    \includegraphics[scale=0.3]{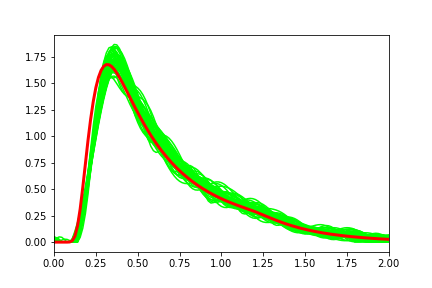}\includegraphics[scale=0.3]{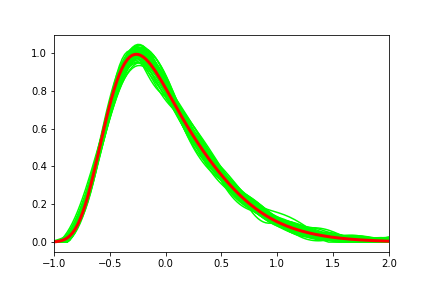}\includegraphics[scale=0.3]{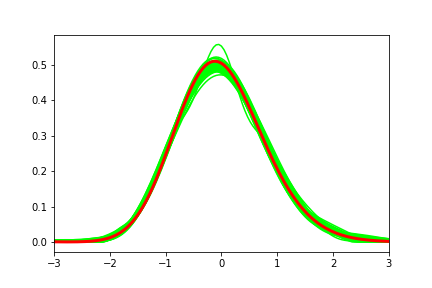}\\
    \includegraphics[scale=0.3]{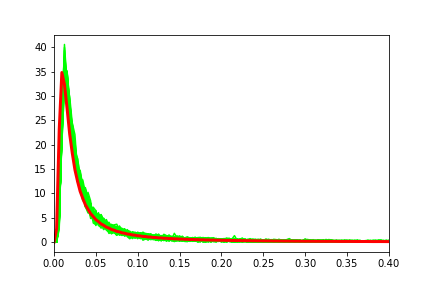}\includegraphics[scale=0.3]{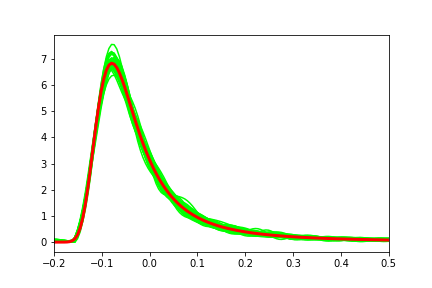}\includegraphics[scale=0.3]{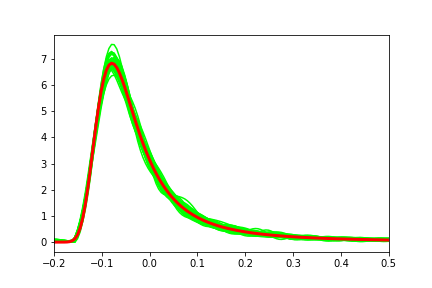}
\caption{Plot of $50$ realisations of $\hat{g}_{\Delta,\hat{m}}$ (green) and an approximation of $g_\Delta$(red) for $\alpha \in \{0.7,1.1,1.7\}$ (columns), $\Delta=1$, (first line) $\Delta=0.1$ (second line), $\Delta=0.01$ (third line), $n=3000$ and $P=2, Q=0, A=0$.}
\label{figure:Stable_NS}
\end{figure}

\begin{figure}[h!]
    \centering
    \includegraphics[scale=0.3]{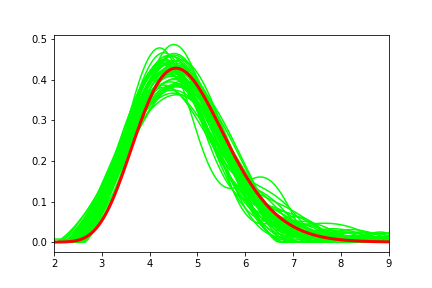}
    \includegraphics[scale=0.3]{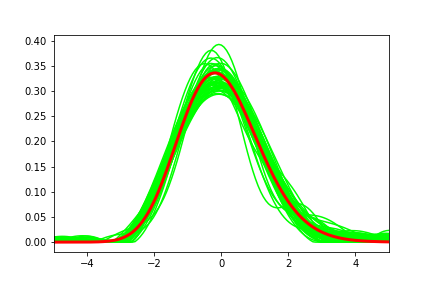}\\
   \includegraphics[scale=0.3]{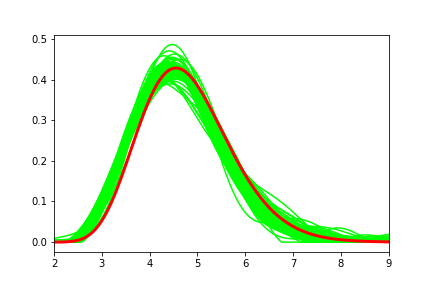}
   \includegraphics[scale=0.3]{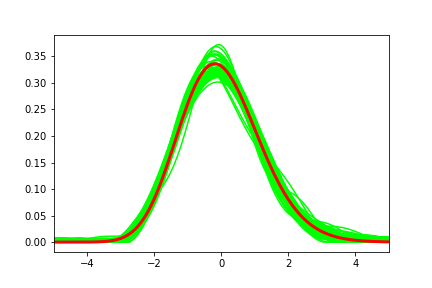}
    \caption{Plot of $50$ realisations of $\hat{g}_{\Delta,\hat{m}}$ (green) and an approximation of $g_\Delta$(red) in the tempered stable case for $\alpha \in \{0.7,1.1\}$, $\Delta=1$, $n=500,1000$ and $P=2, Q=0, A=1$.}
    \label{fig:Tempered_Stable}
\end{figure}

\begin{table}[h!]
\begin{tabular}{cc|cc||cc||cc}
\multicolumn{2}{c}{ } & \multicolumn{2}{c||}{$\Delta=1$} & \multicolumn{2}{c||}{$\Delta=0.1$} & \multicolumn{2}{c}{$\Delta=0.01$} \\
\hline
$\alpha$ & $n$ &  $\tfrac{\nrm{\hat{g}_{\Delta,\hat{m}} - g_\Delta}^2}{\nrm{g_\Delta}^2}$   & $\overline{\hat{m}}$   &$\tfrac{\nrm{\hat{g}_{\Delta,\hat{m}} - g_\Delta}^2}{\nrm{g_\Delta}^2}$   & $\overline{\hat{m}}$   &$\tfrac{\nrm{\hat{g}_{\Delta,\hat{m}} - g_\Delta}^2}{\nrm{g_\Delta}^2}$   & $\overline{\hat{m}}$   \\
\hline
\multirow{6}{*}{0.7}  & \multirow{2}{*}{500} & $4.21 \times 10^{-1}$& 1.57 & $2.30 \times 10^{-2}$& 16.92&$2.65 \times 10^{-2}$ &428.14\\ 
& & {\small{(0.36)}} &  {\small{(0.02)}}& {\small{(0.01)}} & {\small{(3.19)}} & {\small{(0.01)}} & {\small{(87.22)}}\\ \cmidrule{3-8}
& \multirow{2}{*}{1000} & $1.90 \times 10^{-1}$ &1.57 &$1.50 \times 10^{-2}$ & 19.30&$1.41 \times 10^{-2}$ & 524.73\\ 
& & {\small{ (0.15)}} &  {\small{(0.02)}}& {\small{(0.01)}} & {\small{(3.22)}} & {\small{($0.49 \times 10^{-2}$)}} & {\small{(80.92)}}\\ \cmidrule{3-8}
& \multirow{2}{*}{10000} &$2.90 \times 10^{-2}$ &1.58 &$2.23 \times 10^{-3}$ &28.74 & $1.82 \times 10^{-3}$&821.53\\ 
& & {\small{(0.02)}} &  {\small{(0.03)}}& {\small{($0.75 \times 10^{-3}$)}} & {\small{(3.47)}} & {\small{($0.54 \times 10^{-3}$)}} & {\small{(89.02)}}\\ \hline
\hline
\multirow{6}{*}{1.1}  & \multirow{2}{*}{500} &$1.18 \times 10^{-1}$ & 1.60& $1.12 \times 10^{-2}$& 7.61& $1.24 \times 10^{-2}$&62.05\\ 
& & {\small{(0.13)}} &  {\small{(0.10)}}& {\small{(0.01)}} & {\small{(1.40)}} & {\small{($0.58 \times 10^{-2}$)}} & {\small{(11.73)}}\\ \cmidrule{3-8}
& \multirow{2}{*}{1000} & $5.74 \times 10^{-2}$&1.58 & $6.51 \times 10^{-3}$& 8.47& $7.35 \times 10^{-3}$&67.98\\ 
& & {\small{(0.06)}} &  {\small{(0.06)}}& {\small{($3.40 \times 10^{-3}$)}} & {\small{(1.53)}} & {\small{($0.35 \times 10^{-2}$)}} & {\small{(11.06)}}\\ \cmidrule{3-8}
& \multirow{2}{*}{10000} & $6.48 \times 10^{-3}$ & 1.58 &$7.50 \times 10^{-4}$ & 11.07&$7.62 \times 10^{-4}$ &91.09\\ 
& & {\small{(0.05)}} &  {\small{(0.05)}}& {\small{($0.30 \times 10^{-3}$)}} & {\small{(1.24)}} & {\small{($0.30 \times 10^{-3}$)}} & {\small{(8.76)}}\\ \hline
\hline
\multirow{6}{*}{1.7}  & \multirow{2}{*}{500} & $7.78 \times 10^{-2}$& 1.58& $7.19 \times 10^{-3}$& 3.04&$7.42 \times 10^{-3}$ &11.54\\ 
& & {\small{(0.06)}} &  {\small{(0.06)}}& {\small{($0.68 \times 10^{-3}$)}} & {\small{(0.67)}} & {\small{($5.80 \times 10^{-3}$)}} & {\small{(2.20)}}\\ \cmidrule{3-8}
& \multirow{2}{*}{1000} &$3.72 \times 10^{-2}$ &1.57 & $3.83 \times 10^{-3}$& 3.04& $4.48 \times 10^{-3}$&12.97\\ 
& & {\small{(0.03)}} &  {\small{(0.04)}}& {\small{($0.28 \times 10^{-3}$)}} & {\small{(0.24)}} & {\small{($2.80 \times 10^{-3}$)}} & {\small{(2.97)}}\\ \cmidrule{3-8}
& \multirow{2}{*}{10000} & $3.90 \times 10^{-3}$&1.57 & $8.6 \times 10^{-4}$& 3.73& $1.25 \times 10^{-3}$&14.89\\ 
& & {\small{(0.01)}} &  {\small{(0.04)}}& {\small{$(0.50 \times 10^{-3})$}} & {\small{(0.30)}} & {\small{$(0.60 \times 10^{-3})$}} & {\small{(2.08)}}\\ 
\end{tabular}
\caption{ Mean and standard deviation of the relative $L^2$ risk and selected cut-off $\hat{m}$ for $100$ estimators in the $\alpha$-stable symmetric case $P=1, Q=1, A=0$. }
\label{table:stable_Sym}
\end{table}

\begin{table}[h!]
\begin{tabular}{cccc||cc||cc}
\multicolumn{2}{c}{ } & \multicolumn{2}{c||}{$\Delta=1$} & \multicolumn{2}{c||}{$\Delta=0.1$} & \multicolumn{2}{c}{$\Delta=0.01$} \\
\hline
$\alpha$ & $n$ &  $\tfrac{\nrm{\hat{g}_{\Delta,\hat{m}} - g_\Delta}^2}{\nrm{g_\Delta}^2}$   & $\overline{\hat{m}}$   &$\tfrac{\nrm{\hat{g}_{\Delta,\hat{m}} - g_\Delta}^2}{\nrm{g_\Delta}^2}$   & $\overline{\hat{m}}$   &$\tfrac{\nrm{\hat{g}_{\Delta,\hat{m}} - g_\Delta}^2}{\nrm{g_\Delta}^2}$   & $\overline{\hat{m}}$   \\
\hline
\multirow{6}{*}{0.7}  & \multirow{2}{*}{500} &$3.48 \times 10^{-1}$& 1.58& $3.76 \times 10^{-2}$& 15.91 &$1.11 \times 10^{-1}$ &461.14\\ 
& & {\small{(0.27)}} &  {\small{ (0.03)}}& {\small{($1.50 \times 10^{-2}$)}} & {\small{(3.01)}} & {\small{($0.28 \times 10^{-1}$)}} & {\small{(76.86)}}\\ \cmidrule{3-8}
& \multirow{2}{*}{1000} &$1.62 \times 10^{-1}$ &1.58 &$2.67 \times 10^{-2}$ & 19.31 &$9.90 \times 10^{-2}$ & 527.81\\ 
& & {\small{(0.11)}} &  {\small{(0.03)}}& {\small{($0.80 \times 10^{-2}$)}} & {\small{(2.69)}} & {\small{($0.17 \times 10^{-1}$)}} & {\small{(81.93)}}\\ \cmidrule{3-8}
& \multirow{2}{*}{10000} & $5.14 \times 10^{-2}$& 1.58& $1.90 \times 10^{-2}$& 29.16&$8.66 \times 10^{-2}$ &810.72\\ 
& & {\small{(0.03)}} &  {\small{(0.05)}}& {\small{$0.24 \times 10^{-2}$}} & {\small{(3.21)}} & {\small{($0.64 \times 10^{-2}$)}} & {\small{(78.32)}}\\ \hline
\hline
\multirow{6}{*}{1.1}  & \multirow{2}{*}{500} & $8.55 \times 10^{-2}$&1.59 &$9.96 \times 10^{-3}$ & 7.53& $1.06 \times 10^{-2}$&60.95\\ 
& & {\small{(0.10)}} &  {\small{(0.05)}}& {\small{($0.60 \times 10^{-3}$)}} & {\small{(1.41)}} & {\small{($0.60 \times 10^{-2}$)}} & {\small{(9.94)}}\\ \cmidrule{3-8}
& \multirow{2}{*}{1000} & $4.11 \times 10^{-2}$& 1.58& $5.57 \times 10^{-3}$ &8.32 & $5.96 \times 10^{-3}$ &68.05\\ 
& & {\small{(0.05)}} &  {\small{ (0.05)}}& {\small{($0.30 \times 10^{-3}$)}} & {\small{(1.19)}} & {\small{($0.27 \times 10^{-2}$)}} & {\small{(11.17)}}\\ \cmidrule{3-8}
& \multirow{2}{*}{10000} & $5.14 \times 10^{-3}$& 1.60&$7.51 \times 10^{-4}$ & 11.14&$8.09 \times 10^{-4}$ &89.81\\ 
& & {\small{($0.40 \times 10^{-2}$)}} &  {\small{(0.10)}}& {\small{($0.30 \times 10^{-3}$)}} & {\small{(1.41)}} & {\small{($0.04 \times 10^{-2}$)}} & {\small{(9.62)}}\\ \hline
\hline
\multirow{6}{*}{1.7}  & \multirow{2}{*}{500} & $7.58 \times 10^{-2}$& 1.59 & $7.29 \times 10^{-3}$& 3.02& $8.14 \times 10^{-3}$&11.69\\ 
& & {\small{(0.07)}} &  {\small{(0.07)}}& {\small{($5.90 \times 10^{-3}$)}} & {\small{(0.58)}} & {\small{($0.75 \times 10^{-2}$)}} & {\small{(2.47)}}\\ \cmidrule{3-8}
& \multirow{2}{*}{1000} & $3.28 \times 10^{-2}$&1.58 &$3.95 \times 10^{-3}$ & 3.12 & $4.51 \times 10^{-3}$&12.20\\ 
& & {\small{(0.02)}} &  {\small{(0.05)}}& {\small{($2.50 \times 10^{-3}$)}} & {\small{(0.38)}} & {\small{($0.26 \times 10^{-2}$)}} & {\small{(1.72)}}\\ \cmidrule{3-8}
& \multirow{2}{*}{10000} & $4.24 \times 10^{-3}$& 1.59&$8.82 \times 10^{-4}$ & 3.82& $1.28 \times 10^{-3}$ &15.04\\ 
& & {\small{($0.34 \times 10^{-2}$)}} &  {\small{(0.08)}}& {\small{($0.50 \times 10^{-3}$)}} & {\small{(0.44)}} & {\small{($0.05 \times 10^{-2}$)}} & {\small{(1.57)}}\\ 
\end{tabular}
\caption{ Mean and standard deviation of the relative $L^2$ risk and selected cut-off $\hat{m}$ for $100$ estimators in the $\alpha$-stable non symmetric case $P=2, Q=0, A=0$.  }
\label{table:stable_NS}
\end{table}

\begin{table}
\begin{tabular}{c|cc||cc}
\multicolumn{1}{c}{ }&\multicolumn{2}{c||}{$\alpha=0.7$ } & \multicolumn{2}{c}{$\alpha=1.1$ }\\
\hline
 $n$   & $\frac{\nrm{\hat{g}_{\Delta,\hat{m}} - g_\Delta}^2}{\nrm{g_\Delta}^2}$   & $\overline{\hat{m}}$ &$\frac{\nrm{\hat{g}_{\Delta,\hat{m}} - g_\Delta}^2}{\nrm{g_\Delta}^2}$   & $\overline{\hat{m}}$                                                 \\ \hline
\multirow{2}{*}{500} & $2.78 \times 10^{-2}$ & 2.44& $9.43 \times 10^{-3}$&1.87\\
& {\small{(0.016)}} &  {\small{(0.40)}} &  {\small{(0.007)}} & {\small{(0.27)}}
\\ \hline
\multirow{2}{*}{1000} &  $1.89 \times 10^{-2}$& 2.63& $4.62 \times 10^{-3}$&1.99\\
&  {\small{(0.019)}} &  {\small{(0.30)}} &  {\small{(0.002)}} & {\small{ (0.24)}} \\
\end{tabular}
\caption{ Mean and standard deviation of the relative $L^2$ risk and selected cut-off $\hat{m}$ for $100$ estimators when $\Delta=1$ $n=500$ in the tempered stable non symmetric case $P=2, Q=0, A=1$.  }
\label{table:Tempered_stable_NS}
\end{table}

Figure \ref{figure:Stable_NS} illustrates the behaviour of our procedure for $\alpha$-stable processes, while Figure \ref{fig:Tempered_Stable} demonstrates it for tempered stable processes, displaying the outcomes of 50 estimators. Tables \ref{table:stable_Sym} and \ref{table:stable_NS} present the estimated relative $L^{2}$ risks, along with their standard deviations, for the symmetric and non-symmetric $\alpha$-stable cases, respectively. Additionally, Table \ref{table:Tempered_stable_NS} provides the same metrics for the tempered case. These estimations are obtained via Monte Carlo simulation over 100 iterations. The tables also include the mean and standard deviation of the selected $\hat{m}$ values.

A preliminary observation is that, as expected, we observe improvements in both the graphs and the risks as $n$ increases. Below we discuss the influence of $\alpha$ and $\Delta$.
We recall that the rate is provided in \eqref{eq:rate}, which increases as $\Delta$ decreases and increases as $\alpha$ decreases.

\paragraph*{Discussion on the influence of $\alpha$} Given $n$ and $\Delta$ we observe in Tables \ref{table:stable_Sym} and \ref{table:stable_NS} that indeed the relative $L^2$ error decreases with $\alpha$. The estimator is more accurate when the jump activity is higher. 
This phenomenon can be interpreted in various ways. In the case of high frequency observations, the fact that better results are obtained as $\alpha$ increases can be explained by the observation that as $\alpha$ becomes larger, $\|f_\Delta-g_\Delta\|^2$ gets smaller. For instance, using similar arguments as those employed in the proof of Theorem \ref{thm2}, one can show that $\|f_\Delta-g_\Delta\|^2\leq e^{2\lambda\Delta}\frac{\Gamma(1/\alpha)}{(2\alpha M)^{\frac{1}{\alpha}}}\Delta^{2-\frac{1}{\alpha}}$.
Consequently, as $\alpha$ increases, the inverse problem starts resembling more and more a direct problem, yielding more informative observations for the estimation of $g_\Delta$.
In the scenario where $\Delta$ is fixed, this phenomenon can be attributed to the fact that in such a regime the Gaussian approximation of small jumps is better, and the approximation improves as $\alpha$ increases, see e.g. Theorem 1 in \cite{Carpentier_2021}. With increasing $\alpha$, we move closer to a parametric problem, implying a potential improvement in convergence rates.

Moreover in both Figures \ref{figure:Stable_NS} and \ref{fig:Tempered_Stable} we notice that the larger $\alpha$, the larger the support of $g_{\Delta}$ gets. When $\alpha$ gets close to 2 the estimated curve is visually similar to the density of a Gaussian random variable. Finally, we observe that the supremum of $g_\Delta$ decreases with $\alpha$, which is consistent with the behavior in $\alpha$ of the bound given in Lemma \ref{lemmapicard} which allows to derive 
\begin{align}\label{UPg}\|{g_\Delta}\|_{\infty} \leq \tfrac{1}{2\eps}+ \tfrac1\alpha\tfrac{\pi}{2(\Delta M)^{\frac1\alpha}}\Gamma\left(\tfrac{1}{\alpha},\frac{\Delta M}{\eps^{\alpha}}\right).
\end{align}

\paragraph*{Discussion on the influence of $\Delta$}
In Tables \ref{table:stable_Sym} and \ref{table:stable_NS}, as well as in Figure \ref{figure:Stable_NS}, the risks appear to be smaller for $\Delta=0.1$ than when $\Delta=1$ or $\Delta=0.01$. This might seem surprising, given that in \eqref{eq:rate}, the rate seems to be smaller for larger values of $\Delta$. However, this observation is valid only when we neglect the quantity $e^{4\lambda\Delta}$. When we take this into account, we notice that the $L^2$ bound is indeed smaller for $\Delta=0.1$ compared to both $\Delta=0.01$ and $\Delta=1$, which is consistent with what is observed in Tables \ref{table:stable_Sym} and \ref{table:stable_NS}. Additionally, note that Figure \ref{figure:Stable_NS} illustrates that as $\Delta$ decreases, the estimated $\|\hat g_{\Delta,\hat m}\|_{\infty}$ increases, accordingly to the bound provided in \eqref{UPg}.

\paragraph{Comparison between $\alpha$-stable and tempered stable Lévy processes} The comparison of Table \ref{table:Tempered_stable_NS} with Table \ref{table:stable_NS} for $\Delta=1$ reveals better performances of the estimator in the tempered stable case compared to the $\alpha$-stable one. Due to  computational time limitations the case $\alpha=1.7$ is not considered in Table \ref{table:Tempered_stable_NS}. 
Heuristically, one of the distinctive characteristics of tempered stable processes is the reduced concentration of mass on big jumps, due to the exponential term in the Lévy measure. This means that we are more likely to observe small jumps in comparison to the $\alpha$-stable case. Consequently,  it was expected that the estimator would exhibit better performances in the tempered setting. This is emphasized in the case where $\Delta=1$ where the compensation of big jumps complicates the distinction between small and large jumps.

\subsection{Estimation in presence of a Brownian part}
Finally, we illustrate how the presence of a Gaussian component affects the numerical results. We demonstrate this phenomenon by perturbing the same $1$-stable Lévy process with the addition of a Brownian part $\sigma B$ for different values of $\sigma$. We observe in Table \ref{table:stable_brownian_Sym} and Figure \ref{fig:stable_brownian_sym} that the relative $L^2$ error deteriorates as $\sigma$ increases, as noted in Theorem \ref{thm:MB} and Corollary \ref{cor:rateB}.

\begin{table}[h!]
\begin{tabular}{cc||cc||cc||cc}
\multicolumn{2}{c||}{$\sigma=0$} &\multicolumn{2}{c||}{$\sigma=0.2$} & \multicolumn{2}{c||}{$\sigma=0.5$} & \multicolumn{2}{c}{$\sigma=1$} \\
\hline
   $\tfrac{\nrm{\hat{g}_{\Delta,\hat{m}} - g_\Delta}^2}{\nrm{g_\Delta}^2}$   & $\overline{\hat{m}}$   & $\tfrac{\nrm{\hat{g}_{\Delta,\hat{m}} - g_\Delta}^2}{\nrm{g_\Delta}^2}$   & $\overline{\hat{m}}$   &$\tfrac{\nrm{\hat{g}_{\Delta,\hat{m}} - g_\Delta}^2}{\nrm{g_\Delta}^2}$   & $\overline{\hat{m}}$   &$\tfrac{\nrm{\hat{g}_{\Delta,\hat{m}} - g_\Delta}^2}{\nrm{g_\Delta}^2}$   & $\overline{\hat{m}}$   \\
   \hline
 $1.72 \times 10^{-2}$& 1.597 & $1.91 \times 10^{-2}$& 1.596 &$2.11 \times 10^{-2}$ &1.582 & $9.97 \times 10^{-2}$& 1.589\\ 
 {\small{(0.016)}} &  {\small{(0.08)}}& {\small{(0.022)}} & {\small{(0.08)}} & {\small{(0.020)}} & {\small{(0.04)}} & {\small{(0.12)}}  & {\small{(0.06)}} \\
 \end{tabular}
\caption{ Mean and standard deviation of the estimated relative $L^2$ risk and selected cut-off $\hat{m}$ from 100 estimators for $\Delta=1$, $n=5000$, $\alpha=1$, $P=1, Q=1, A=0$.   }
\label{table:stable_brownian_Sym}
\end{table}

\begin{figure}[h!]
\center
\includegraphics[width=4.7cm]{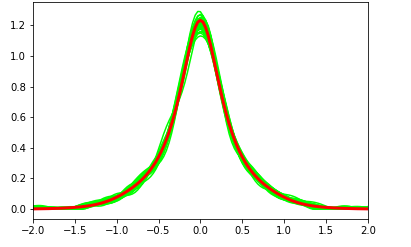}
\includegraphics[width=4.7cm]{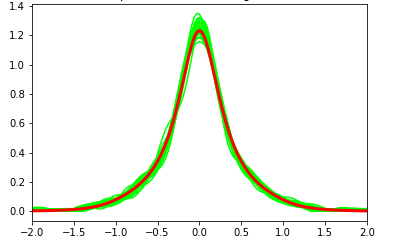}
\includegraphics[width=4.7cm]{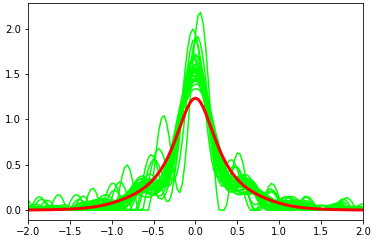}

\caption{Plot of $30$ realisations of $\hat{g}_{\Delta,\hat{m}}$ (green) using the same symmetric $1$-stable process (with $P=1,Q=1,A=0$) perturbed by a Brownian motion with $\sigma \in \{0, 0.2,0.5\}$ and an approximation of  $g_\Delta$ (red) by Fourier inversion with $\Delta=0.1$, $n=5000$.}
\label{fig:stable_brownian_sym}
\end{figure}

\section{Proofs}

\subsection{Proof of Theorem \ref{thm2}}

To control the integrated $L^{2}$-risk introduce the notation $f_{\Delta,m}$, the uniquely defined function with Fourier transform $\mathcal F f_{\Delta,m}= (\mathcal F f_{\Delta})\mathds{1}_{[-m,m]}$, where $f_{\Delta}$ is the density of $X_{\Delta}$. We write the decomposition 
 \begin{align*}\E[\|\tilde{g}_{\Delta,m}-g_{\Delta}\|^2]
 = \|f_{\Delta,m}-g_{\Delta}\|^2+\frac{1}{2\pi}\int_{- m}^{ m} \E[|\hat\phi_{X_{\Delta}}(u)-\phi_{X_{\Delta}}(u)|^2]d u.\end{align*}  The second variance term is easily bounded as in the proof of Theorem \ref{thm1} by $\frac{e^{4}m}{\pi n}$, under the assumption $\lambda \Delta\le 1$. The first term is  a bias term for which we can write 
 \begin{align}
 \|f_{\Delta,m}-g_{\Delta}\|^2&\le 2\|f_{\Delta,m}-g_{\Delta,m}\|^2+2\|g_{\Delta,m}-g_{\Delta}\|^2\nonumber.
 \end{align} 
 An upper bound for $\|g_{\Delta,m}-g_{\Delta}\|^2$ is provided  in \eqref{eq:biasF}.
By means of Plancherel's equality and \eqref{eq:minor}, it holds:
 \begin{align*}
\|f_{\Delta,m}-g_{\Delta,m}\|^2&=\frac1{2\pi}\int_{[-m,m]}|\phi_{X_{\Delta}}(u)|^{2}\left|1-\frac1{\phi_{X^{B}_{\Delta}}(u)}\right|^{2}du.
 \end{align*} 
 Recalling that $(\phi_{X_\Delta^B}(u))^{-1}=e^{\lambda\Delta(1-\phi_{Y_1}(u))}$, by means of the mean value theorem we get
 $$\bigg|1-\frac{1}{\phi_{X_\Delta^B}(u)}\bigg|\leq e^{2\lambda\Delta}-1\leq 2\lambda \Delta e^{2\lambda \Delta}\leq  2\lambda \Delta e^2.$$
Hence,  \begin{align*}
\|f_{\Delta,m}-g_{\Delta,m}\|^2&\le\frac{2e^{4}(\lambda\Delta)^2}{\pi}\int_{\R}|\phi_{X_{\Delta}}(u)|^{2}du.
 \end{align*} 
 Furthermore, using \eqref{eq:Picard1} in Lemma \ref{lemmapicard} and the inequality $\Gamma(s,x)\leq 2^se^{-x/2}\Gamma(s)$ for $s,x>0$, we obtain 
 \begin{align*}
 \|\phi_{X_{\Delta}}\|_2^{2}&\leq \int_{|u|\geq \frac{\pi}{2\eps}}  |\phi_{X_{\Delta}}(u)|^2du+\int_{|u|\leq \frac{\pi}{2\eps}} |\phi_{X_{\Delta}}(u)|^2du\leq 2\int_{\frac{\pi}{2\eps}}^{\infty} e^{-\frac{2^{\alpha+1}M}{\pi^{\alpha}}u^{\alpha}\Delta}du+ \frac{\pi}{\eps}\\
 &= \Delta ^{-\frac1\alpha}\frac{\pi }{2^{1+\frac1\alpha}\alpha M^{\frac1\alpha}}\Gamma\left(\frac1\alpha, \frac{2M\Delta}{\eps^{\alpha}}\right)+\frac\pi\eps\leq \Delta ^{-\frac1\alpha}\frac{\pi }{2^{1+\frac1\alpha}\alpha M^{\frac1\alpha}}e^{-\frac{M\Delta}{2\eps^{\alpha}}}2^{\frac1\alpha}\Gamma(\tfrac1\alpha)+
 \frac\pi\eps\\
 &\leq K'\left( \Delta^{-\frac1\alpha}e^{-\kappa\frac{\Delta}{\eps^{\alpha}}}+\tfrac1\eps\right),
 \end{align*}
 for some positive constants $K',\, \kappa$, depending on $\alpha$ and $M$. Gathering both terms we derive that there exist two positive constants $K$ and $\kappa$ such that
  \begin{align}
\|f_{\Delta,m}-g_{\Delta,m}\|^2&\le K\lambda^2\left(\Delta^{2-\frac{1}{\alpha}}e^{-\kappa\frac{\Delta}{\eps^{\alpha}}}+\Delta^{2}\eps^{-1}\right). \label{eq:biasFG0}
 \end{align} 
Collecting all terms, we derive the desired result.
\subsection{Proof of Theorem \ref{thmLB}}
We begin by showing that 

$$\inf_{\hat g_{\Delta}}\sup_{g_{\Delta}\in \mathcal A_{M,\alpha_0}}\E\|\hat g_{\Delta}-g_{\Delta}\|^{2}\geq \inf_{\hat f_{\Delta}}\sup_{f_{\Delta}\in {\mathcal S}_{\alpha_0}}\E\|\hat f_{\Delta}-f_{\Delta}\|^{2},$$
where $\mathcal S_{\alpha_0}$ denotes the class of symmetric $\alpha$-stable densities with $\alpha\in[\alpha_0,2)$ and  with characteristic function $e^{-\Delta |u|^{\alpha}}$ and the second infimum is taken over all possible estimators $\hat f_\Delta$ of the density $f_\Delta$ of $X_{\Delta}$ (with respect to the Lebesgue measure). 
Let $\mu =\delta_{0}+{\rm Leb}$ and denote by $\mathcal{F_{\mu}}$ the set of densities with respect to the measure $\mu$ and by $h_\Delta$ the density of the law of $X_\Delta^B$ with respect to $\mu$. 
 We make two observations: first $\|h_{\Delta}\|_{1,\mu}=\int h_{\Delta}(x) \mu(dx)=1\geq \|h_\Delta\|_1$ and second the Young inequality allows to write that $\|g\star h\|_{2}\le \|g\|_{2}\|h\|_{1}$ for $g\in L_{2}$ and $h\in L_{1}$. Consequently, we derive that
\begin{align*}
\inf_{\hat g_{\Delta}}\sup_{g_{\Delta}\in \mathcal A_{M,\alpha_0}}\E\|\hat g_{\Delta}-g_{\Delta}\|^{2}&=\inf_{\hat g_{\Delta}}\sup_{g_{\Delta}\in \mathcal A_{M,\alpha_0}}\sup_{h_{\Delta}\in \mathcal F_{\mu}	}\E\|\hat g_{\Delta}-g_{\Delta}\|^{2}\|h_{\Delta}\|_{1,\mu}^{2}\\ &\ge \inf_{\hat g_{\Delta}}\sup_{g_{\Delta}\in \mathcal A_{M,\alpha_0}}\sup_{h_{\Delta}\in \mathcal F_{\mu}	}\E\|\hat g_{\Delta}\star h_{\Delta}-g_{\Delta}\star h_{\Delta}\|^{2}
\\ &= \inf_{\hat g_{\Delta}}\sup_{{f_{\Delta}, f_{\Delta} = g_{\Delta}\star h_{\Delta}}\atop{g_{\Delta\in \mathcal A_{M,\alpha_0},\,h_{\Delta}\in \mathcal F_{\mu}	}}}\E\|\hat g_{\Delta}\star h_{\Delta}-f_{\Delta}\|^{2}
\\ &\ge \inf_{\hat f_{\Delta}}\sup_{f_{\Delta}\in {\mathcal S}_{\alpha_0}}\E\|\hat f_{\Delta}-f_{\Delta}\|^{2},
\end{align*} 
where we used that if $f_{\Delta}\in \mathcal S_{\alpha_0}$ then for \begin{align}
M\le F(\alpha_{0}):
=\begin{cases}
\left (\frac{2-\alpha_{0}}{\alpha_{0}}\cos(\frac{\pi\alpha_{0}}{2})\Gamma({1-\alpha_{0}})\right)^{-1} & \text{if }\alpha_{0}\ne 1,\\
\frac2\pi  & \text{if }\alpha_{0}= 1,\\
\end{cases}\label{eq:MajM}
\end{align} 
there exists $g_\Delta\in \mathcal A_{M,\alpha_0}$ and $h_{\Delta}\in \mathcal F_{\mu}$ such that $f_\Delta=g_\Delta\star h_\Delta$. Indeed, if $f_\Delta$ is the marginal density of $X_\Delta$, its characteristic function is $e^{-\Delta|u|^{\alpha_{0}}}$ and it has a L\'evy density $p(x) = \frac{P}{|x|^{1+\alpha_0}}\ind_{x\neq 0}$ with $P=\left(2\cos\left(\frac{\pi\alpha_{0}}{2} \right) \alpha_0^{-1} \Gamma(1-\alpha_0)\right)^{-1}$ (see \cite{belomestny2015levy} p.215). 
The proof is then concluded by using Steps 2,3 and 4 of the proof of Theorem 3 in \cite{DJM} where it is shown that
$$\inf_{\hat f_{\Delta}}\sup_{f_{\Delta}\in {\mathcal S}_{\alpha_0}}\E\|\hat f_{\Delta}-f_{\Delta}\|^{2}\geq \frac{c_{0}}{n\Delta^{\frac{1}{\alpha_0}}},$$
 for $c_{0}$ a strictly positive constant only depending on $\alpha_0$.

\subsection{Proof of Theorem \ref{thm:adapt}}
Firstly, we observe that
$$\gamma_{n}(\hat g_{\Delta,\hat m})+\mbox{pen}(\hat m)\le\gamma_{n}(\hat g_{\Delta, m})+\mbox{pen}( m). $$ Moreover
\begin{align*}\gamma_{n}(t)-\gamma_{n}(s)&=\|t-g_{\Delta}\|^{2} -\|s-g_{\Delta}\|^{2}-2\langle g_{\Delta} ,t-s\rangle-\frac1\pi\langle \hat \phi_{Z_{\Delta}}, \mathcal F(t-s)\rangle\\&=\|t-g_{\Delta}\|^{2} -\|s-g_{\Delta}\|^{2}-2\nu_{n}(t-s),\end{align*}
where
$$\nu_{n}(t)=\frac1{2\pi}\langle \hat \phi_{Z_{\Delta}}-\phi_{Z_{\Delta}},\mathcal F(t)\rangle=\frac{1}{2\pi n}\sum_{j=1}^{n}\int \left(\frac{e^{iu(X_{j\Delta}-X_{(j-1)\Delta})}}{\phi_{X^{B}_{\Delta}}(u)}-\frac{\E[e^{iuX_{\Delta}}]}{\phi_{X^{B}_{\Delta}}(u)}\right)\mathcal F t(-u)du,$$
 using Plancherel. 
Combining these results, we derive  
\begin{align*}&\|\hat g_{\Delta,\hat m}-g_{\Delta}\|^{2}\\&\le \|\hat g_{\Delta,m}-g_{\Delta}\|^{2}+2\nu_{n}(\hat g_{\Delta,\hat m}-\hat g_{\Delta, m})+\mbox{pen}(m)-\mbox{pen}(\hat m)\\
&=  \|\hat g_{\Delta,m}-g_{\Delta}\|^{2}+2\|\hat g_{\Delta,\hat m}-\hat g_{\Delta, m}\|\nu_{n}\left(\frac{\hat g_{\Delta,\hat m}-\hat g_{\Delta, m}}{\|\hat g_{\Delta,\hat m}-\hat g_{\Delta, m}\|}\right)+\mbox{pen}(m)-\mbox{pen}(\hat m)\\
&\le  \|\hat g_{\Delta,m}-g_{\Delta}\|^{2}+2\|\hat g_{\Delta,\hat m}-\hat g_{\Delta, m}\|\sup_{t\in S_{m}+S_{\hat m},\| t\|=1}\nu_{n}(t)+\mbox{pen}(m)-\mbox{pen}(\hat m)\\
&\le  \|\hat g_{\Delta,m}-g_{\Delta}\|^{2}+\frac14\|\hat g_{\Delta,\hat m}-\hat g_{\Delta, m}\|^{2}+4\sup_{t\in S_{m}+S_{\hat m},\| t\|=1}\nu_{n}(t)^{2}+\mbox{pen}(m)-\mbox{pen}(\hat m)\\
&\le\frac32\|\hat g_{\Delta,m}-g_{\Delta}\|^{2}+\frac12\|\hat g_{\Delta,\hat m}- g_{\Delta}\|^{2}+4\left(\sup_{t\in S_{m}+S_{\hat m},\| t\|=1}\nu_{n}(t)^{2}-p(m,\hat m)\right)_{+}\\ & \quad\quad+4p(m,\hat m)+\mbox{pen}(m)-\mbox{pen}(\hat m),
\end{align*} where $p(m,m')=\frac4\pi e^{4\lambda\Delta}(m\vee m')/n$ is fixed by applying the Talagrand inequality to $\nu_{n}$ (see the following Lemma \ref{lem:Tal}). Note that $S_{m}+S_{m'}=S_{m\vee m'}.$

\begin{lemma}\label{lem:Tal}
There exists a positive constant $C$ such that 
\begin{align*}
\mathbb{E}\left(\sup_{t\in {S_{m\vee \hat m},\ \|t\|=1}}|\nu_{n}(t)|^2-4e^{4\lambda\Delta}\frac{m\vee \hat m}{\pi n}\right)_+ &\le \frac{C}{n}.
\end{align*}
\end{lemma}

Plugging this result in above inequalities implies that
\begin{align*}
\frac12\E[\|\hat g_{\Delta,\hat m}-g_{\Delta}\|^{2}]\le \frac32&\left(\E[\|\hat g_{\Delta,m}-g_{\Delta}\|^{2}]+\mbox{pen}(m)\right)+ \frac{4C}{n}\\&+\E[4p(m,\hat m)-\mbox{pen}(\hat m)]-\frac12\mbox{pen}( m),
\end{align*} using that for $\kappa>32/(3\pi)$
$$ 4p(m,\hat m)-\mbox{pen}(\hat m)-\frac12\mbox{pen}( m)= \frac{e^{4\lambda\Delta}}n\left(\frac{16}\pi {(m\vee \hat m)}-\kappa\left(\frac 12m+\hat m\right)\right)\le 0.$$ 
The proof is completed by taking the infimum over $m$.

\begin{proof}[Proof of Lemma \ref{lem:Tal}] We apply the Talagrand inequality recalled in Lemma \ref{lem:T} in the Appendix. Note that we can write 
$$ \nu_{n}(t)=\frac1n\sum_{j=1}^{n}\big(f_{t}(X_{j\Delta}-X_{(j-1)\Delta})-\E[f_t(X_{\Delta})]\big)$$
where for $t\in S_{m\vee m'}$, $f_{t}(x)=\frac1{2\pi}\int_{-m\vee m'}^{m\vee m'}\frac{e^{iux}}{\phi_{X^{B}_{\Delta}}(u)}\mathcal Ft(-u)du.$  For that we compute the three positive constants $M$, $H$ and $v$ introduced in Lemma \ref{lem:T}. First note that as $\|t\|=1$ we get using Cauchy-Schwarz and \eqref{eq:minor} that
\begin{align*}
\sup_{t\in S_{m\vee m'},\|t\|=1 }\| f_{t}\|_{\infty}&\le \left\|\frac{1}{\phi_{X^{B}_{\Delta}}}\right\|_{\infty}\frac{1}{2\pi}\sqrt{2(m\vee m')}\sqrt{\int_{-m\vee m'}^{m\vee m'}|\mathcal{F}t(u)|^{2}du}\\&\le e^{2\lambda\Delta}\frac{\sqrt{m\vee m'}}{\sqrt{\pi}}=:M.
\end{align*} Using similar arguments  we get
\begin{align*}
 \E\left(\sup_{t\in S_{m\vee m'},\|t\|=1}|\nu_{n}({t})|\right)^{2}&\le \E\left[\sup_{t\in S_{m\vee m'},\|t\|=1}\nu_{n}({t})^{2}\right]\\&\le\frac1{2\pi}\int_{-m\vee m'}^{m\vee m'}\frac{\E[|\hat \phi_{X_{\Delta}}(u)-\phi_{X_{\Delta}}(u)|^{2}]}{|\phi_{X^{B}_{\Delta}}(u)|^{2}}du\\
&\le \frac{e^{4\lambda\Delta}}{\pi n}(m\vee m')=:H^{2}.
\end{align*} Finally for the last term, following \cite{comte2010nonparametric} we use the basis representation of the estimator to compute $v^{2}$. Indeed, using  the basis \eqref{eq:basis} it holds $t=\sum_{j\in\Z}b_{j}\psi_{j,m\vee m'}$ with $b_{j}=\langle t,\psi_{j,m\vee m'}\rangle$ such that $\sum_{j\in \Z}b_{j}^{2}=1,$ and we can write
\begin{align*}
\V(f_{t}(X_{\Delta}))&\le \frac{1}{4\pi^{2}}\iint_{-m\vee m'}^{m\vee m'}\frac{\phi_{X_{\Delta}}(u-v)}{\phi_{X^{B}_{\Delta}}(u)\phi_{X^{B}_{\Delta}}(-v)}\mathcal Ft(-u)\mathcal Ft(v)dudv\\
&=\frac{1}{4\pi^{2}}\sum_{j,k\in\Z}b_{j}b_{k}\iint_{-m\vee m'}^{m\vee m'}\frac{\phi_{X_{\Delta}}(u-v)}{\phi_{X^{B}_{\Delta}}(u)\phi_{X^{B}_{\Delta}}(-v)}\mathcal{F}\psi_{j,m\vee m'}(-u)\mathcal{F}\psi_{k,m\vee m'}(v)dudv\\
&\hspace{-1cm}\le \frac{1}{4\pi^{2}}\sqrt{\sum_{j,k\in\Z}b_{j}^{2}b_{k}^{2}\sum_{j,k\in\Z}\left|\iint_{-m\vee m'}^{m\vee m'}\frac{\phi_{X_{\Delta}}(u-v)}{\phi_{X^{B}_{\Delta}}(u)\phi_{X^{B}_{\Delta}}(-v)}\mathcal{F}\psi_{j,m\vee m'}(-u)\mathcal{F}\psi_{k,m\vee m'}(v)dudv\right|^{2}}\\
&= \frac{1}{4\pi^{2}}\sqrt{\iint_{-m\vee m'}^{m\vee m'}\left|\frac{\phi_{X_{\Delta}}(u-v)}{\phi_{X^{B}_{\Delta}}(u)\phi_{X^{B}_{\Delta}}(-v)}\right|^{2}dudv}\\ &\le \frac{e^{4\lambda\Delta}}{4\pi^{2}}\sqrt{\iint_{-m\vee m'}^{m\vee m'}\left|{\phi_{X_{\Delta}}(u-v)}\right|^{2}dudv}\le \frac{\sqrt{\pi}e^{4\lambda\Delta}}{2\pi^{2}}\sqrt{{m\vee m'}}\| f_{X_{\Delta}}\|,
\end{align*} where we used  at the third line the Cauchy-Schwarz inequality on the index $\lambda=(j,k)$ and at the penultimate equality that for a bi-variate function $\phi(u,v)$ its norm can be computed as $\|\phi\|^{2}=\sum_{j,k\in \Z}\langle \phi,\psi_{j,m\vee m'}\otimes \psi_{k,m\vee m'}\rangle^{2}$.
Therefore,
\begin{align*}
\sup_{t\in S_{m\vee m'},\|t\|=1}\frac1n\sum_{j=1}^{n}\V(f_{t}(X_{j\Delta}-X_{(j-1)\Delta}))&\le \frac{\sqrt{\pi}e^{4\lambda\Delta}}{2\pi^{2}}\sqrt{{m\vee m'}}\| f_{X_{\Delta}}\|=:v^{2}.
\end{align*} It follows from the Talagrand inequality (see Lemma \ref{lem:T}, $\delta=1/2$) that there exist positive constants $C_{1},\ C_{2}$ and $C_{3}$ such that
\begin{eqnarray*}
\mathbb{E}\left(\sup_{t\in {S_{m\vee m'},\ \|t\|=1}}|\nu_{n}(t)|^2-4e^{4\lambda\Delta}\frac{m\vee m'}{\pi n}\right)_+ &\hspace{-0.1cm}\leq   C_{1}\left(\frac {\sqrt{m\vee m'}}n e^{-C_{2}  \sqrt{m\vee m'}}+
\frac{m\vee m'}{n^2} e^{-C_{3}\sqrt{n}}\right).
\end{eqnarray*}
Finally,  \begin{align*}
&\mathbb{E}\left(\sup_{t\in {S_{m\vee \hat m},\ \|t\|=1}}|\nu_{n}(t)|^2-4e^{4\lambda\Delta}\frac{m\vee \hat m}{\pi n}\right)_+\hspace{-0.3cm}\\& \leq  \sum_{m'=1}^{n} \mathbb{E}\left(\sup_{t\in {S_{m\vee  m'},\ \|t\|=1}}|\nu_{n}(t)|^2-4e^{4\lambda\Delta}\frac{m\vee  m'}{\pi n}\right)_+\\
\quad&\le \frac{C_{1}}n\sum_{m'=1}^{n}\left( {\sqrt{m\vee m'}} e^{-C_{2}  \sqrt{m\vee m'}}+
\frac{m\vee m'}{n} e^{-C_{3}\sqrt{n}}\right)\le \frac{C}{n}.
\end{align*} \end{proof}

\subsection{Proof of Theorem \ref{thm:MB}}
The proof is similar to the one of Theorem \ref{thm1}, the only difference lies in the treatement of the variance term.
This term now writes
 \begin{align*}\E[|\check\phi_{Z_{\Delta}}(u)-\phi_{Z_{\Delta}}(u)|^2]&=\frac{\E[|\hat\phi_{X_{\Delta}}(u)-\phi_{X_{\Delta}}(u)|^2]}{|\phi_{{X^{B}_{\Delta}}}(u)|^{2}|\phi_{\sigma W_\Delta}(u)|^{2}}= \frac{1-|\phi_{{X_{\Delta}}}(u)|^{2}}{|\phi_{{X^{B}_{\Delta}}}(u)|^{2}}\frac{e^{\sigma^2 \Delta u^2}}{n},
  \end{align*} which yields the following bound
 \begin{align}\E[\|\check{g}_{\Delta,m}-g_{\Delta,m}\|^2]\le \frac{e^{4\lambda\Delta}}{\pi}\frac{\int_{0}^m e^{\sigma^2 \Delta u^2}du}{n}.
 \label{eq:VarF_brownian}
 \end{align}
A change of variable completes the proof.

\subsection{Proof of Corollary \ref{cor:rateB}}
The upper bound in Theorem \ref{thm:MB} is a square bias and variance decomposition, for $\alpha=1$ we compute $m^{\star}$ that minimizes the quantity:
\begin{align}\label{biasvar}
C\Delta^{-1}\Gamma\left(1,c\Delta m\right)+\frac{e^{4\lambda\Delta}}{\pi} \frac {\int_{0}^{m\sqrt{\Delta}\sigma} e^{z^2}dz}{n\sqrt{\Delta}\sigma}&= \frac{C}{\Delta}e^{-c\Delta m}+\frac{e^{4\lambda\Delta}}{\pi} \frac {\int_{0}^{m\sqrt{\Delta}\sigma} e^{z^2}dz}{n\sqrt{\Delta}\sigma}\\ &\nonumber=:B^{2}(m)+V(m). \end{align}
Differentiating in $m$, we get that  $m^{\star}$ is solution of the following equivalent equations
\begin{align}
  -C c e^{-c\Delta m} + \frac{e^{4\lambda\Delta}}{\pi n} e^{\sigma^2 \Delta {m}^2}&=0,  \label{eq:m_alpha1_log}\quad 
e^{\sigma^2 \Delta m^2 + c \Delta m} = c_\lambda n,
\end{align} where  $c_\lambda = \pi cC e^{-4\lambda \Delta}$. Considering the positive root of this equation we derive that
\begin{align*}
m^{\star} = \frac{-c + \sqrt{c^2 + 4\sigma^2 \Delta^{-1} \log(c_\lambda n)}}{2\sigma^2}= \frac{2\Delta^{-1} \log(c_\lambda n)}{c + \sqrt{c^2 + 4\sigma^2 \Delta^{-1} \log(c_\lambda n)}}.
\end{align*} 

\paragraph{Case $\Delta\sigma\ge a>0$, $\sigma$ fixed}
As $\sigma\Delta^{1/2}m^\star\rightarrow \infty$, we use that $\int_{0}^x e^{\beta y^{2}} dy \underset{x\rightarrow \infty}{\sim} \frac{1}{2\beta x} e^{\beta x^2}$, joined with \eqref{biasvar} and \eqref{eq:m_alpha1_log} to derive that the variance term is asymptotically equivalent to 
\begin{align*}
V(m^{\star}):= \frac{e^{4\lambda\Delta}}{\pi \sigma \Delta^{1/2} n} \int_{0}^{\sigma \Delta^{1/2} m^{\star}} e^{z^2} dz \underset{n\rightarrow \infty}{\sim} \frac{e^{4\lambda\Delta}}{2\pi \sigma^2 \Delta m^{\star}n} e^{\sigma^2 \Delta{m^\star}^2} &= \frac{Cc}{2\sigma^2\Delta m^{\star}} e^{-c \Delta m^{\star}}\\&= \frac{c B^2(m^{\star})}{2\sigma^2 m^{\star}}.
\end{align*}
As $m^{\star} \rightarrow \infty$ and $\sigma$ is fixed, the $L^2$ error bound is asymptotically bias dominated, therefore the bias term dictates the rate (see also \cite{MR2354572,MR2742504}). 
In order to derive the corresponding convergence rate which will be of order ${\Delta^{-1}}e^{-c\Delta m^{\star}}$, we compute the order of the exponent $c\Delta m^{\star}$. 
Using the explicit form of $m^\star$, we write for $R^2= \sigma^2\Delta^{-1} \log(c_\lambda n) \rightarrow \infty$\begin{align*}
c\Delta m^\star &= \frac{2c \log(c_\lambda n)}{c + \sqrt{c^2 + 4\sigma^2 \Delta^{-1} \log(c_\lambda n)}} 
= \frac{c\Delta^{1/2} \log(c_\lambda n)^{1/2}}{\sigma} \frac{1}{\tfrac{c}{2R} + \sqrt{1+ \big(\tfrac{c}{2R}\big)^{2}}}\\
&= \frac{c \Delta^{1/2} \log(c_\lambda n)^{1/2}}{\sigma} - \frac{c^2 \Delta}{2\sigma} + o(1),
\end{align*}
implying that $c\Delta m^\star - \left(\frac{c \Delta^{1/2} \log(c_\lambda n)^{1/2}}{\sigma} - \frac{c^2 \Delta}{2\sigma} \right) \underset{n\rightarrow \infty}{\rightarrow} 0.$ Plugging this in the square bias term,  we derive the announced convergence rate: 
$
\E\left[ \nrm{\check{g}_{\Delta,m^{\star}} - g_\Delta}^2\right] \leq K\ e^{-\kappa \sqrt{\log n}}$
where $K$ and $\kappa$ are positive constants.

\paragraph{Case where $\sigma$ goes to 0 rapidly} Note that the integrand of the variance term writes \begin{align*}
\sigma \Delta^{1/2} m^{\star} = \frac{2\sigma\Delta^{-1/2} \log(c_\lambda n)}{c + \sqrt{c^2 + 4\sigma^2 \Delta^{-1} \log(c_\lambda n)}} = \frac{2{R\log(c_\lambda n)^{1/2} }}{c + \sqrt{c^2 + 4R^2}}.
\end{align*}  If $\sigma^{2}\Delta^{-1} (\log n)^{2}\le 1$, then $\sigma\Delta^{1/2}m^\star$ remains bounded.  In fact, it either converges to $0$ or is bounded away from $0$. If it tends to 0, taking advantage of \eqref{eq:m_alpha1_log}, \eqref{biasvar} and the property $\int_{0}^{x} e^{y^2} dy \underset{x\rightarrow 0}{\sim} xe^{x^2}$, we find that the variance term is equivalent to:
\begin{align*}
V(m^{\star}) = \frac{e^{4\lambda\Delta}}{\pi \sigma \Delta^{1/2} n} \int_{0}^{\sigma \Delta^{1/2}m^\star} e^{z^2} dz \sim \frac{e^{4\lambda\Delta}}{\pi n} m^{\star} e^{\sigma^2 \Delta {m^{\star}}^2} = Cc m^{\star}  e^{-c\Delta m^{\star}} = c\Delta m^{\star} B^2(m^{\star}).
\end{align*}
Since $R=\sigma^2\Delta^{-1} \log(c_\lambda n) \rightarrow 0$ we get   $\Delta m^{\star} =  \frac{2\log(c_\lambda n)}{c + \sqrt{c^2 + 4R^2}} \rightarrow \infty $. 
Otherwise, $\sigma \Delta^{1/2} m^\star$ can be considered constant. The bias term can be expressed as:
\begin{align*}
B^2(m^\star) = C \Delta^{-1} e^{-c\Delta m^\star} = \frac{e^{4\lambda\Delta}}{\pi n \Delta} e^{\sigma^2 \Delta {m^\star}^2} = \mathcal{O}\left(\frac{1}{n\Delta} \right),
\end{align*}
whereas, using that $\sigma \leq\frac{ \sqrt{\Delta}}{\log n}$, we have

\begin{align*}
V(m^\star) \geq \frac{e^{4\lambda\Delta}}{\pi}\int_0^{m\sqrt{\Delta}\sigma}e^{z^2}dz\frac{\log n}{n\Delta}= \mathcal{O}\left( \frac{\log n}{n\Delta }\right).
\end{align*}
Hence, in both cases, the $L^2$ error is asymptotically variance-dominated, leading to the following bound for the $L^2$ rate:
$$ \E\left[ \nrm{\check{g}_{\Delta,m^{\star}} - g_\Delta}^2\right] \le\frac{e^{4\lambda\Delta}e^{\sigma^2 \Delta {m^{\star}}^2} }{\pi}\frac {\Delta m^{\star}}{ n\Delta}  \le K' \frac{\log n}{ n\Delta},$$
where $K'$ is a positive constant.

\begin{appendix}
\section*{Appendix}
\noindent {\bf Smoothness of the Lévy density.} The result below directly follows from Lemma 2.3 in \cite{picard1997}.
\begin{lemma}\label{lemmapicard}
Let $X$ be a Lévy process satisfying \eqref{Ass:p}. 
Then, \begin{align}\label{eq:Picard1}
|\phi_t(u)|\leq e^{-\frac{2^{\alpha}M}{\pi^{\alpha}}|u|^{\alpha}t},\quad \forall |u|\geq \frac{\pi}{2\eps},\ \forall t>0.
\end{align}
Furthermore, $X_t$ has a smooth density $g_t$ with all the derivates uniformly bounded:
\begin{align}\label{eq:Picard2}
\sup_{x\in\R} |g_t^{(k)}(x)|\leq\frac1{\pi(k+1)}\left(\frac{\pi}{2\eps}\right)^{k+1}+ \frac1\alpha\left(\frac{\pi}{2(tM)^{\frac1\alpha}}\right)^{k+1}\Gamma\left(\frac{k+1}{\alpha},\frac{tM}{\eps^{\alpha}}\right),\quad \forall k\geq 0, \ \forall t>0.
\end{align}
\end{lemma}

\begin{proof}[Proof of Lemma \ref{lemmapicard}]
The Lévy-Khintchine formula allows to write for $u\ne 0$
\begin{align*}
|\phi_{t}(u)|=\exp\left(t\int_{\R}(\cos(u x)-1)\nu(dx)\right)\le\exp\left(t\int_{-\frac{\pi}{2|u|}}^{\frac{\pi}{2|u|}}(\cos(u x)-1)\nu(dx)\right).
\end{align*} Using that for $|x|\le \frac{\pi}{2}$ it holds $-\frac{x^{2}}{2}\le \cos(x)-1\le -\frac{4x^{2}}{\pi^{2}}$ and \eqref{Ass:p}, for all $|u|\ge \frac{\pi}{2\eps}$ we get: 
$$ |\phi_{t}(u)|\le\exp\left(-\frac{4tu^{2}}{\pi^{2}}\int_{-\frac{\pi}{2|u|}}^{\frac{\pi}{2|u|}}x^{2}p(x)dx\right)\le e^{-\frac{2^{\alpha}M}{\pi^{\alpha}}|u|^{\alpha}t}.$$ In particular, this ensures that $u\mapsto u^{k}|\phi_{t}(u)|$ is integrable for any $k\in\N$, that $g_{t}$ is in $C^{k}(\R)$ and for all $x\in\R$ it holds that
\begin{align*}
|g_{t}^{(k)}(x)|&\le \frac1\pi\int_{0}^{\infty}u^{k}|\phi_{t}(u)|du\le\frac1{\pi(k+1)}\left(\frac{\pi}{2\eps}\right)^{k+1}+ \frac1\alpha\left(\frac{\pi}{2(tM)^{\frac1\alpha}}\right)^{k+1}\Gamma\left(\frac{k+1}{\alpha},\frac{tM}{\eps^{\alpha}}\right),
\end{align*} where we split the integral at $\pi/(2\eps)$.

\end{proof}

\noindent {\bf The Talagrand inequality.} The result below follows from the Talagrand concentration inequality given in \cite{klein2005} and arguments in \cite{birge1998} (see the proof of their
Corollary 2 page 354).
\begin{lemma}
\label{lem:T} (Talagrand Inequality) Let $Y_1, \dots, Y_n$ be independent random
variables   and let ${\mathcal F}$ be a countable class of
uniformly bounded measurable functions. Consider $\nu_{n}$, the centered empirical process defined by  $$\nu_{n}(f)=\frac1n\sum_{i=1}^n [f(Y_i)-{\mathbb
E}(f(Y_i))]$$ for $f\in\mathcal{F}$.
Assume there exist three positive constants $M,\ H$ and $v$ such that
$$\sup_{f\in {\mathcal F}}\|f\|_{\infty}\leq M, \;\;\;\;
\mathbb{E}\Big[\sup_{f\in {\mathcal F}}|\nu_{n}(f)|\Big]\leq H,
\quad \sup_{f\in {\mathcal F}}\frac{1}{n}\sum_{k=1}^n{\rm
Var}(f(Y_k)) \leq v^{2}.$$
 Then, for any $\delta>0$ the following holds
\begin{align*}
\mathbb{E}\Big[\sup_{f\in {\mathcal
F}}|\nu_{n}(f)|^2-2(1+2\delta)H^2\Big]_+ &\leq   \frac
4{K_1}\left(\frac {v^{2}}n \exp\left(-K_1\delta \frac{nH^2}{v^{2}}\right) \right.\\&+\left.
\frac{49M^2}{K_1n^2C^2(\delta)} \exp\left(-\frac{K_1
C(\delta)\sqrt{2\delta}}{7}\frac{nH}{M}\right)\right),
\end{align*}
with $C(\delta)=\sqrt{1+\delta}-1$ and $K_1=1/6$.
\end{lemma}
By standard density arguments, this result can be extended to the case where ${\mathcal F}$ is a unit ball of a linear normed space, after checking that $f\mapsto \nu_n(f)$ is continuous and ${\mathcal F}$ contains a countable dense family.

\end{appendix}





%
%
%
%
%

\end{document}